\documentclass[11pt, a4paper, twoside, reqno]{amsart}
\usepackage[utf8]{inputenc}
\usepackage[T1]{fontenc}
\usepackage{lmodern}
\usepackage{microtype}
\usepackage[frak=boondox]{mathalpha}
\usepackage{dsfont}
\usepackage{bbold}
\usepackage{upgreek}
\usepackage{mathrsfs}
\usepackage{euscript}
\usepackage{amssymb}

\usepackage[top=1in, bottom=1in, left=1in, right=1in, headheight=15pt, footskip=40pt]{geometry}
\usepackage{enumitem}
\usepackage{parskip}
\usepackage{float}
\usepackage{caption}
\usepackage{tabto}
\usepackage{color}
\usepackage{mdframed}
\usepackage{hyphenat}
\usepackage{comment}

\usepackage{amsmath, amssymb, amsthm, mathtools}
\usepackage{extarrows}
\usepackage{tikz-cd}
\usepackage[all,cmtip]{xy} 
\usepackage{graphicx}

\usepackage[numbers]{natbib}
\setcitestyle{open={},close={}}
\makeatletter
\renewcommand{\@biblabel}[1]{[#1]\hfill}
\makeatother

\usepackage[hidelinks]{hyperref}
\usepackage[nameinlink]{cleveref}
\newtheoremstyle{mystyle}
  {\topsep}    
  {\topsep}   
  {\itshape}   
  {}           
  {\bfseries}   
  {.}          
  {.5em}        
  {}            

\newtheoremstyle{spacedremark} 
  {\topsep}    
  {\topsep}     
  {\normalfont} 
  {}           
  {\bfseries}   
  {.}          
  {.5em}        
  {}               

\theoremstyle{mystyle}
\newtheorem{thm}{Theorem}[subsection]\crefname{thm}{Theorem}{Theorems}
\newtheorem{lem}[thm]{Lemma}\crefname{lem}{Lemma}{Lemmas}

\newtheorem{prop}[thm]{Proposition}\crefname{prop}{Proposition}{Propositions}
\newtheorem{cor}[thm]{Corollary}\crefname{cor}{Corollary}{Corollaries}
\newtheorem{con}[thm]{Construction}\crefname{con}{Construction}{Construction}
\newtheorem{obs}[thm]{Observation}\crefname{obs}{Observation}{Observations}

\newtheorem{defn}[thm]{Definition}\crefname{defn}{Definition}{Definitions}

\newtheorem{inner-recall-star}{Theorem}
\newenvironment{recall*}[1]
  {\begin{inner-recall-star}[see \Cref{#1}]}
  {\end{inner-recall-star}}
\theoremstyle{spacedremark}
\newtheorem{rem}[thm]{Remark}\crefname{rem}{Remark}{Remarks}
\newtheorem{ex}[thm]{Example}\crefname{ex}{Example}{Examples}
\crefname{coex}{Counterexample}{Counterexamples}

\DeclareMathOperator*{\colim}{colim}

\newcommand{\dcolim}{\varinjlim}
\newcommand{\plim}{\varprojlim}

\DeclareMathAlphabet{\duc}{U}{dutchcal}{m}{n}
\SetMathAlphabet{\duc}{bold}{U}{dutchcal}{b}{n}
\DeclareFontFamily{U}{BOONDOX-calo}{\skewchar\font=45 }
\DeclareFontShape{U}{BOONDOX-calo}{m}{n}{<-> s*[1.0] BOONDOX-r-calo}{}
\DeclareFontShape{U}{BOONDOX-calo}{b}{n}{<-> s*[1.0] BOONDOX-b-calo}{}
\DeclareMathAlphabet{\mal}{U}{BOONDOX-calo}{m}{n}
\SetMathAlphabet{\mal}{bold}{U}{BOONDOX-calo}{b}{n}

\newcommand{\esc}{\EuScript}

\setlist[enumerate]{leftmargin=*, nosep}
\begin{document}

\title[Connectivity of the slice filtration]{Connectivity of the slice filtration}

\author[Dipankar Maity]{Dipankar Maity}
\address{Department of Mathematical Sciences \\
Indian Institute of Science Education and Research (IISER) Mohali \\ Knowledge city, Sector 81, SAS Nagar, Manauli PO 140306 \\
India}
\email{abstracthomotopies@gmail.com}

\begin{abstract}
Using methods similar to Morel's stable connectivity theorem, we prove that, over an arbitrary field, the Tate truncation functors $f_{0/n}$ preserve motivic connectivity of $ S^1$-spectra. An analogous result for complexes yields a Hurewicz theorem for the $L^{p,n}$- and $L_{bir}^n$-localizations over perfect fields. Over such fields, we establish a stronger connectivity property for categories of correspondences, showing that $f_n^\mathcal{C}$ preserves connectivity of motivic spectra with $\mathcal{C}$-transfers, whenever $\mathcal{C}$ satisfies cancellation. We then use the motivic reconstruction theorem to deduce that $f_n$, and consequently $s_n$ and $f_{0/n}$, also preserve connectivity for effective (and thereby, $\mathbb{P}^1$-) motivic spectra. Along the way, we also establish the slice analog of the motivic (effective) reconstruction theorem, as well as the slice analog of motivic $S^1$- and $\mathbb{P}^1$-recognition theorems.
\end{abstract}
\maketitle

\tableofcontents
\section{Introduction}
\subsection{Motivation to the paper}
 
A defining feature of the stable motivic homotopy category $\mathcal{SH}^{S^1}(k)$ (or $\mathcal{SH}^{\mathbb{P}^1}(k)$) is that it possesses two different circle objects: the simplicial circle $S^1$ and the algebraic circle $\mathbb{G}_m$, which yield the famous bigrading on homotopy invariant cohomology theories. This inherent bigrading induces two distinct filtrations.

The first is Morel's Postnikov filtration [\cite{morel2005stable}], governed by $S^1$-suspension, which equips $\mathcal{SH}^{S^1}(k)$ with its standard homotopy $t$-structure. For any object $\mal{X} \in \mathcal{SH}^{S^1}(k)$, this yields the Postnikov tower of connective covers:$$\dots \to \tau_{\ge n+1}\mal{X} \to \tau_{\ge n}\mal{X} \to \tau_{\ge n-1}\mal{X} \to \dots \to \mal{X}$$ whose successive cofibers yield the shifted homotopy sheaves ${\pi}^{nis,sp}_n(\mal{X})[n]$ and the total cofiber yields the postnikov truncation:
$$\tau_{\ge n}\mal{X}  \to \mal{X}\to \tau_{<n}\mal{X}.$$

The second is Voevodsky's slice filtration [\cite{voevodsky2002possible}], [\cite{voevodsky2002open}], governed by the $\mathbb{G}_m$-suspension coordinate. Filtering by the localizing subcategories $\mathcal{SH}^{S^1}(k)\wedge \mathbb{G}^{n}$, this produces the slice tower of tate connective covers for $\mal{X}$:$$\dots \to f_{n+1}\mal{X} \to f_n \mal{X} \to f_{n-1}\mal{X} \to \dots \to \mal{X}$$ whose successive cofibers $s_n \mal{X} \simeq \text{cofib}(f_{n+1}\mal{X} \to f_n \mal{X})$ isolate the motivic weights and whose total quotients are (to be called) Tate quotients:
$$f_n\mal{X}\to \mal{X}\to f_{0/n}\mal{X}.$$

A fundamental structural question concerns the compatibility of these filtrations, specifically, whether the respective truncation and connective covering functors preserve each other's connective subcategories and quotients. While such compatibilities are degenerate in classical topology, the cross-compatibility of these two distinct filtrations in the motivic setting is highly nontrivial and, to some extent, governed by Voevodsky's slice conjectures [\cite{voevodsky2002possible}].

The main objective of this paper is to prove that the slice-truncation functors ($f_{0/n}$) preserve topological connectivity. For $S^1$-spectra over perfect fields, a version of this connectivity was previously established by Bachmann [\cite{MR4173925}, see the first paragraph of \S6] based on ([Footnote 10 in loc.cit.]) the validity of a variant of the slice conjecture proven by Levine in [\cite{levine2008homotopy}]); see \Cref{bachmann gm conservativity} below for a quick summary of Bachmann's argument. Our first main result bypasses this requirement, providing an algebraic proof valid over arbitrary fields. We achieve this by identifying the $n$-th Tate truncation $f_{0/n}$ as a localization at the smooth affine sphere $\mathrm{S}^{n-1}_\mathbb{A} := \mathbb{A}^n \setminus \{0\}$ and then adapting the proof of Morel's stable connectivity theorem [\cite{morel2005stable}] to the Tate setting (as was done in the $\mathbb{P}^1$-local setting by Ayoub in [\cite{ayoub2020P1}]).

As a primary application of these connectivity properties, we obtain various $t$-structures on the slice quotients and the stable birational categories (over perfect fields) of spectra and complexes, and thus obtain a Hurewicz theorem for motivic nullification [\cite{asok2023p}] and $n$-birational localization [\cite{sfbat}]. 

Lifting these connectivity results from $S^1$-spectra directly to the setting of $\mathbb{P}^1$-spectra is hindered by the slice conjectures (which have been verified over perfect fields in multiple works [\cite{levine2008homotopy}], [\cite{bachmann2019voevodsky}]). We provide a structurally independent proof of connectivity for effective spectra (and thereby for $\mathbb{P}^1$-spectra), using the motivic reconstruction theorem (and its slice variant) and verifying a similar connectivity property for motivic $S^1$-spectra with framed transfers.

While obtaining the contents of the last paragraph, we adapt the motivic infinite loop space machine [\cite{elmanto2021motivic}] to the slice filtration: we establish an $S^1$-recognition principle for connective Tate slice quotients and lift it to a $\mathbb{P}^1$-recognition principle. This yields a monoidal equivalence between generalized ($n$-fold) slice [\cite{MR3743071}] quotients in $\mathcal{SH}^{veff}(k)$ and group-like framed motivic spaces whose homotopy groups have vanishing ($n$-fold) $\mathbb{G}_m$-contractions.
 
\subsection{Notations and terminologies}
Throughout this paper, $S$ denotes a qcqs scheme and $k$ denotes a field, which is not assumed to be perfect unless specified. All morphisms of schemes are assumed to be separated.

We freely use the language of $\infty$-categories without choosing a specific model; a standard reference is [\cite{lurie2009higher}]. For an $\infty$-category $\mathcal{C}$, let $\mathcal{P}(\mathcal{C})$ denote the $\infty$-category of space-valued presheaves on $\mathcal{C}$. We set $\mathcal{S}pt(\mathcal{C}) := \mathrm{Stab}(\mathcal{P}(\mathcal{C})) \simeq \mathcal{P}(\mathcal{C}, \mathcal{S}pt)$. For a symmetric monoidal $\infty$-category $\mathcal{D}$, $\mathrm{CMon}(\mathcal{D})$ denotes the category of commutative monoid objects in $\mathcal{D}$, and the superscript $\mathrm{gp}$ indicates group completion.

When $\mathcal{C} = Sm_S$, we abbreviate $\mathcal{P}(Sm_S)$ as $\mathcal{P}(S)$ and $\mathcal{S}pt(Sm_S)$ as $\mathcal{S}pt(S)$. A topology $\tau$ is indicated via subscripts (e.g., $\mathcal{P}_\tau$, $\mathcal{S}pt_\tau$). We denote by $\mathcal{SH}^{S^1}(S)$ the full subcategory of $\mathcal{S}pt_{\mathrm{nis}}(S)$ consisting of $\mathbb{A}^1$-local Nisnevich local objects, with $L_{\mathrm{mot}}^{sp}$ denoting the corresponding localization functor. Combined localizations are designated by a list of appropriate subscripts on $L$. $\mathcal{SH}^{\mathbb{P}^1}(S)$ denotes the monoidal inversion of $\mathbb{G}:= (\mathbb{G}_m,1)$ on $\mathcal{SH}^{S^1}(S)$. This is equivalently the stable $\infty$-category of motivic local $\mathbb{P}^1$-spectra on $Sm_S$. There are canonical adjunctions:
\[
\xymatrix{
\mathcal{H}^{\mathbb{A}^1}(S) \ar@<0.5ex>[r]^-{\Sigma^\infty_{S^1}} & \mathcal{SH}^{S^1}(S) \ar@<0.5ex>[l]^-{\Omega^\infty_{S^1}} \ar@<0.5ex>[r]^-{\Sigma^\infty_{\mathbb{G}}} & \mathcal{SH}^{\mathbb{P}^1}(S) \ar@<0.5ex>[l]^-{\Omega^\infty_{\mathbb{G}}}
}
\]
We denote the composite adjunction by $\Sigma^{\infty}_{\mathbb{P}^1}\dashv \Omega^\infty_{\mathbb{P}^1}$. Furthermore, we denote the essential image of $\Sigma^\infty_\mathbb{G}$ by $\mathcal{SH}^{eff}(S)$. This breaks the second adjunction into two pieces:
\[
\xymatrix{
\mathcal{SH}^{S^1}(S) \ar@<0.5ex>[r]^-{\sigma^\infty_{S^1}} & \mathcal{SH}^{eff}(S) \ar@<0.5ex>[l]^-{\omega^\infty_{S^1}} \ar@<0.5ex>[r]^-{i_0} & \mathcal{SH}^{\mathbb{P}^1}(S) \ar@<0.5ex>[l]^-{r_0}
}
\]

For $n \in \mathbb{N}$, let $$B(n) := \{U \hookrightarrow X \in \mathrm{Mor}(Sm_S) \mid \mathrm{codim}_X(X \setminus U) \ge n + 1\}$$ denote the set of $n$-birational open immersions. Here, a dense open immersion $U \hookrightarrow X$ is $n$-birational if $U$ contains all $x \in X$ with $\mathrm{codim}_X(x) \le n$. We denote by $\mathcal{H}^n(S)$ the localization of $\mathcal{H}^{\mathbb{A}^1}(S)$ at $B(n)$, with corresponding localization functor $L_{\mathrm{bir}}^n$. Similarly, $L^{p,q}$ denotes the $S^{p,q}$-nullification functor (as in [\cite{asok2023p}]), given by localization at the projection maps $\{S^{p,q} \times U \to U\}_{U \in Sm_S}$).

We write $\pi_n$ for the presheaf of homotopy groups, and $\pi_n^{\mathrm{nis}}$ for its Nisnevich sheafification. As usual, $\pi_i^{\mathbb{A}^1} := \pi_i^{\mathrm{nis}} L_{\mathrm{mot}}$. More generally, for a localization functor $L_X$, we set $\pi_i^X := \pi_i^{\mathrm{nis}} L_X$. We put the superscript $sp$ to indicate the stabilized version of each of these.

$\mathcal{H}^{fr}(S)$ and $\mathcal{SH}^{fr,S^1}(S)$ stands for the motivic localization of the $\infty$-category of presheaves on the $\infty$-category of framed correspondences [\cite{elmanto2021motivic}] of spaces and spectra respectively. More generally, if $\mathcal{C}$ is an $\infty$-category of correspondences over $S$ (in the sense of [\cite{bachmann2021cancellation}]), then we denote by $\mathcal{H}^{\mathcal{C}}(S)$ and $\mathcal{SH}^{\mathcal{C},S^1}(S)$ for the motivic localization of presheaves of on $\mathcal{C}$ of spaces and spectra respectively.
\subsection{Details of the paper}

The main objective of this paper, as indicated in the introduction, is to show that slice-truncation functors preserve topological connectivity, in analogy with their topological counterparts (where these hold almost tautologically). We divide the proof into two cases: (S) $S^1$-spectra and (P) $\mathbb{P}^1$-spectra. For (S), our approach adapts Morel's stable connectivity theorem [\cite{morel2005stable}]; this is carried out in \S 2. For (P), we present two independent methods:
\begin{enumerate}
\item combining the $S^1$-case with the slice conjectures;     
\item using the reconstruction theorem alongside an $S^1$-connectivity result obtained for framed $S^1$-spectra.
\end{enumerate}
We elaborate on each of these approaches below.

\textbf{(S)}\textit{\textbf{Slice Connectivity for $S^1$-spectra.}} In Section 2, we recall the slice filtration of $S^1$-spectras and construct explicit models for $f_{0/n}$ for every $n\geq 0$. The idea is that the $n$-th Tate truncation is equivalent to localization at the smooth affine sphere $\mathrm{S}^{n-1}_\mathbb{A}:=\mathbb{A}^{n}\setminus 0$:

\begin{recall*}{formula for a1 sn nis localization}
Over an arbitrary qcqs base scheme $S$, there is an equivalences of localization functors (for all $m\geq 0 $) $$L_{mot}^{sp}\simeq\Phi_{\mathbb{A}^1}^\infty L_{nis}^{sp},\quad f_{0/n}\simeq \dcolim_{m}\big( L_{mot}^{sp}\Phi_{\mathrm{S}^{n-1}_\mathbb{A}}\big )^m$$ where the notation $\Phi_A$ is from \Cref{model for loc in stab} and is computed in $Spt^{S^1}(S)$ for $A=\Sigma^\infty_+\mathrm{S}^{n-1}_\mathbb{A}$.
\end{recall*}
Using this along with some ideas from [\cite{ayoub2020P1}] and [\cite{morel2005stable}] we then show that:
\begin{inner-recall-star}[see \Cref{slice connectivity}, \Cref{slices takes connected to connective} and \Cref{zero slice connectivity}]
Over any field $f_{0/n}:\mathcal{SH}^{S^1}(k)\to \mathcal{SH}^{S^1}(k)$ preserves connectivity. In particular $s_0=f_{0/1}$ preserves connectivity. It follows that the higher Tate slices $s_n$ take connected spectra to connective spectra. 
\end{inner-recall-star}

It is well known that connectivity results of this type suffice to construct appropriate $t$-structures on localized categories, which we formulate in an abstract setting at the beginning of \S 4. Combined with the Tate connectivity theorems, this yields a $t$-structure on the slice localization, induced from the standard motivic homotopy $t$-structure [\cite{morel2005stable}]. To state these results, we define $n$-spherical motivic sheaves of abelian groups as those strictly $\mathbb{A}^1$ invariant sheaves of abelian groups whose deloopings are $\mathrm{S}^{n-1}_\mathbb{A}$-local.

   \begin{recall*}{t structure on n spherical motivic spectras}
      Let $k$ be a field. The $n$-th tate truncated category $\mathcal{SH}^{S^1}/f_{n}$ have accessible complete t-structures. The inclusion functor $\mathcal{SH}^{S^1}/f_n\subset \mathcal{SH}^{S^1} $ is $t$-exact. The heart of this $t$-structure is the category of strictly $(n-1)$-spherical invariant strictly $\mathbb{A}^1$-invariant sheaves of abelian groups (when $k$ is a perfect field, this is equivalent to the kernel of the functor $(-)_{-{n}}: Ab_k^{\mathbb{A}^1}\to Ab_k^{\mathbb{A}^1}$).
 \end{recall*}
 
We denote by $Ab^{\mathbb{A}^1}_k/f_n$ the kernel from the above theorem (i.e., full subcategory of $Ab^{\mathbb{A}^1}_k$ with trivial $n$-fold $\mathbb{G}_m$-contraction). The following standard counterpart of [\cite{morel2005stable}, Theorem 6.2.7] is then immediate:
\begin{recall*}{characterizing n bir spectra}
    Let $\esc{X}$ be a motivic spectrum over a perfect field. The following are then equivalent:
    \begin{enumerate}
        \item $\esc{X}$ is the $n$-th tate truncation of a motivic space
        \item For all $i$ the nisnevich connective covers $\tau^{nis}_{\geq i}\esc{X}$ are $n$-tate truncation of a motivic spectra.
        \item For all $i$ we have $\pi_i^{nis}\esc{X}\in Ab^{\mathbb{A}^1}_k/f_n$.
    \end{enumerate}
\end{recall*}

To conclude this section, we discuss the homological algebraic topology of Tate truncation of motivic complexes. We do not include proofs, since they follow similarly to the case of spectra, but we state and prove the Hurewicz theorem for motivic nullification:
  \begin{recall*}{slice hurewicz theorem}
Let $k$ be a perfect field and $\esc{X}$ be a pointed $n$-birationally $(m-1)$-connected space ($m\geq 0$).  Let $$\pi_i^{p,n}\esc{X}:=\pi_i^{nis}L^{p,n}\esc{X},,\quad H^{0/n}_i\esc{X}:=H_i^{nis}f_{0/n}^{ch}\Gamma \esc{X}.$$ 
Then the Hurewicz map $$\pi^{p,n}_m\esc{X}\to H^{0/n}_m\esc{X}$$ is an isomorphism for $m\geq \max{(p-n,2)}$; for $1\leq m<p-n$ it is the universal strictly $\mathbb{A}^1$-invariant, strictly $n$-spherical (i.e., having trivial $n$-fold $\mathbb{G}_m$ contraction) sheaf of abelian groups associated to the sheaf of groups $\pi^{p,n}_m\esc{X}$, and for $m=0$ it is the free such sheaf of abelian groups generated by $\pi_0^{p,n}$.
\end{recall*}
In the third section, we study some consequences of \S2. As an example, we introduce the $n$ birational localization of the $S^1$ stable motivic homotopy category and then show that over perfect fields it agrees with the $n$th Tate truncated motivic category. This immediately yields:
\begin{recall*}{n birational stable connectivity}
    Over a perfect field, $L_{bir}^{n,sp}$ preserves connectivity.
\end{recall*}
Actually, all the results from \S2 apply immediately to $n$-birational localization over perfect fields. Throughout \S3.1, we state them without proof.

In the 4th section, we conduct a similar study for the motivic homotopy category construction of generalized correspondences (a suitable notion of correspondences that encompasses most examples of categories of correspondences, including Voevodsky's finite correspondence, Milnor-Witt correspondence, and framed correspondence, etc.; see [\cite{MR4324462}]). Instead of using the same techniques as in \S2, we adopt a different method specific to cancellative correspondences. In fact, we show that there is an easy formula for $f_n$ in such cases:

\begin{recall*}{model for connective cover for canc corr}
 Suppose $\mathcal{C} $ is a category of correspondence over $k$ satisfying cancellation \textup{[\cite{bachmann2021cancellation}, Definition 2.11]}. Then, for every $n\geq 0$, there is a canonical equivalence:
    $$f_{n}^{\mathcal{C}}\simeq (-)^{\mathbb{G}_\mathcal{C}^n}\otimes \mathbb{G}_\mathcal{C}^n.$$    
    \end{recall*}

Using standard properties of the functors for adding and forgetting transfers, we show that:
\begin{recall*}{cancellative slice connectivity}
 Let $\mathcal{C}$ be a cancellative category of correspondences over a perfect field $k$. Then, for every $n\geq 0$, the functors $f_n^\mathcal{C}$, $f_{0/n}^\mathcal{C}$ and $s_n^\mathcal{C}$ are right $t$-exact as endo functors of $\mathcal{SH}^{\mathcal{C},S^1}(k)$ for the homotopy $t$-structure generated by smooth schemes.
\end{recall*}
\textbf{(P)}\textit{\textbf{Slice Connectivity for $\mathbb{P}^1$-spectra.}} As indicated at the beginning, we prove the connectivity property for the effective (and $\mathbb{P}^1$) slice filtration in two different ways. In \S 3.2, we use the validity of the slice conjectures over perfect fields to reduce the question to $S^1$-spectra, concluding via the results of \S 2. Since the findings in that subsection rely on the slice conjectures and are less general, we do not highlight them further here.

The main result for effective and $\mathbb{P}^1$-spectra is established in \S 5.2, where we combine the slice connectivity for framed motives from \S 4 (\Cref{fr slice connectivity}) with the slice variant of the motivic reconstruction theorem (\Cref{slice reconst}) to yield the desired connectivity result:
\begin{inner-recall-star}[see \Cref{eff fn preserves conn} for the effecive case, and \Cref{p1 fn preserves conn} for the $\mathbb{P}^1$ case]
  Let $k$ be a perfect field and $n\geq 0$ a natural number. Then for the homotopy $t$-structure on $\mathcal{SH}^{eff}(k)$ (and $\mathcal{SH}^{\mathbb{P}^1}(k)$), the slice endo functors enjoy the following properties:
\begin{enumerate}
    \item The colocalization functor $f_{n}$ is $t$-exact.
     \item The functor $s_{n}$ is right $t$-exact.
     \item The functor $f_{0/n}$ is right $t$-exact. 
\end{enumerate} 

\end{inner-recall-star}
In addition to proving connectivity, this paper adapts the motivic infinite loop space machine to the slice filtration. In \S 5.1, we establish an $S^1$-recognition principle for connective Tate slice quotients, in the spirit of [\cite{elmanto2021motivic}, \S 3.1]:

\begin{recall*}{slice recognition over perfect}
    Over a perfect field, there is an equivalence 
    $$\mathrm{B}^\infty_{{nis}} : \mathrm{CMon}_{n}(\mathcal{H}^{\mathbb{A}^1}(k))^{{gp}} \leftrightarrows \mathcal{SH}^{S^1}(k) / f_{n} : \Omega^\infty_{S^1}$$
    where the left-hand side consists of $\mathbb{A}^1$-invariant Nisnevich local grouplike commutative monoids $X$ whose $\pi_0^{\mathbb{A}^1} X$ is strongly $\mathbb{A}^1$-invariant and $(\pi_i^{\mathbb{A}^1} X)_{-n} = 0$ (i.e., $\pi_i^{\mathbb{A}^1} X \in {Ab}^{\mathbb{A}^1}_k/f_n$), and the right-hand side is the full subcategory of $\mathcal{SH}^{S^1}(k)$ generated under colimits by $\Sigma^\infty_+ X / f_n$ for $X\in Sm_k$.
\end{recall*}
In \S5.2 we rewrite the $\mathbb{P}^1$-recognition principle for the slice filtration:
\begin{inner-recall-star}[\Cref{p1 recognition for slices}, \Cref{rhs of p1 slice recognition}]
    Let $k$ be a perfect field. Then the motivic recognition equivalence restricts to an equivalence:
    $$\gamma_*\Sigma_{fr}^{\infty}:\mathcal{H}^{fr}_n(k)^{gp}\leftrightarrows\mathcal{SH}^{veff}(k)/\tilde{f}_n:\Omega^\infty_\mathbb{P}\gamma^*$$ 
    where the left hand is the full subcategory of $\mathcal{H}^{fr}(k)$ consisting of those $\esc{Y}\in \mathcal{H}^{fr}(k)$ such that for all $i\geq 0$ $(\pi_i^{nis}(\gamma_*\esc{Y}))_{-n}=0$ while the right hand side is the generalized slice \textup{[\cite{MR3743071}]} quotient of  $\mathcal{SH}^{veff}(k)$.
\end{inner-recall-star}

\section{Connectivity of the slice filtration of motivic \texorpdfstring{$S^1$}{} spectra}
Let $Spt := \mathrm{Stab}(Spc)$ denote the $\infty$-category of spectra. Classically, the Postnikov slice filtration on $Spt$ is generated by the spheres $S^n$, and the corresponding slice functors are intrinsically connectivity-preserving. In the motivic setting, Thom spaces introduce an additional weight grading, yielding Voevodsky's motivic slice filtration [\cite{voevodsky2002possible}], [\cite{voevodsky2002open}]. A central question is whether the motivic slice functors along Thom spheres preserve topological connectivity. While this holds over perfect base fields [\cite{MR4173925}, Corollary 6.2], the existing proof relies essentially on the slice conjectures (see \Cref{bachmann gm conservativity}). This section establishes this connectivity preservation over arbitrary base fields for $S^1$-spectra, entirely independent of the slice conjectures.

\subsection{Slice filtration of motivic \texorpdfstring{$S^1$}{} spectras}\mbox{}
We begin by recalling the construction of the slice filtration for motivic $S^1$-spectra. While originally introduced over fields in [\cite{voevodsky2002possible}], we formulate the construction here over an arbitrary base scheme.

\underline{\textbf{Motivic $S^1$ spectras.}} 

Let $S$ be a Qcqs scheme. The stable homotopy category of $S^1$-spectras over $S$ is defined as:
\begin{flalign*}
    \mathcal{S}pt^{S^1}(S)&:=\mathcal{P}(Sm_S,Spt)\\
    &\simeq \mathrm{Stab}(\mathcal{P}(S))\equiv \mathrm{Stab}(\mathcal{P}(Sm_S,Spc)).
\end{flalign*}
We denote by $\mathcal{S}pt^{S^1}_{nis}(S)$ the full subcategory of $\mathcal{S}pt^{S^1}(S)$ consisting of presheaves satisfying Nisnevich descent. Since the Nisnevich topology is a cd topology, a presheaf $\mal{F}\in \mathcal{S}pt^{S^1}(S)$ satisfies Nisnevich descent if and only if, for every Nisnevich cd square $Q$, the square $\mal{F}(Q)$ is a pullback square of spectra:
\[Q : \quad
\vcenter{\hbox{$
\xymatrix@R=2pc@C=2pc{
W \ar[r] \ar[d] & V \ar[d]^p \\
U \ar[r]_j & X
}
$}}
\quad \xrightarrow{\quad \mal{F} \quad} \quad
\vcenter{\hbox{$
\xymatrix@R=2pc@C=2pc{
\mal{F}(X) \ar[r] \ar[d] & \mal{F}(U) \ar[d] \\
\mal{F}(V) \ar[r] & \mal{F}(W)
}
$}}
\quad : \mal{F}(Q)\]
Since filtered colimits commute with finite limits in the $\infty$-category $Spt$, we immediately obtain the following observation:\begin{obs}
    The inclusion $\mathcal{S}pt^{S^1}_{nis}(S)\subset \mathcal{S}pt^{S^1}_{}(S)$ is closed under filtered colimits.
\end{obs}

Since $Spt:=\mathrm{Stab}(Spc)$, there is an equivalent way of describing $\mathcal{S}pt^{S^1}_{nis}(S)$, namely $$\mathcal{S}pt^{S^1}_{nis}(S)\simeq \mathrm{Stab}(\mathcal{P}_{nis}(S))\equiv \mathrm{Stab}(L_{nis}\mathcal{P}(Sm_S,Spc)).$$ 
That is $\mathcal{S}pt^{S^1}_{nis}(S)$ is the colimit of the following tower in $Pr^L$:
$$\mathcal{P}_{nis}(S)_\bullet\xrightarrow{\Sigma}\mathcal{P}_{nis}(S)_\bullet\xrightarrow{\Sigma}\mathcal{P}_{nis}(S)_\bullet\xrightarrow{\Sigma}\cdots$$
Therefore, this comes with canonical adjunctions:
\begin{flalign*}
    \Sigma^\infty_+:\mathcal{P}_{nis}(S)&\leftrightarrows \mathcal{S}pt^{S^1}_{nis}(S):\Omega^\infty\\
  L^{sp}_{nis} : \mathcal{S}pt^{S^1}(S)&\mathrel{\substack{\hookleftarrow\\[-0.4ex] \rightarrow}}\mathcal{S}pt^{S^1}_{nis}(S)
\end{flalign*}

Finally, the motivic stable homotopy category of $S^1$-spectra, denoted $\mathcal{SH}^{S^1}(S)$, is defined as the localization $L^{sp}_{mot}\mathcal{S}pt^{S^1}_{nis}(S)$, i.e., the localization of $\mathcal{S}pt^{S^1}_{nis}(S)$ at the (essentially small) set of $\mathbb{A}^1$-projections:
$$\{\Sigma^\infty_+{\mathbb{A}^1_X}\to \Sigma^\infty_+ X\}_{X\in Sm_S}.$$ Equivalently, this is the stabilization of the motivic homotopy category $$\mathcal{H}(S):=L_{\mathbb{A}^1}L_{nis}\mathcal{P}(S)$$, defined and studied in [\cite{morel19991}]. In other words, $\mathcal{SH}^{S^1}(S)$ is the inverse limit of the following tower in $Cat_\infty$:
$$\mathcal{H}(S)_\bullet \xrightarrow{\Sigma}\mathcal{H}(S)_\bullet\xrightarrow{\Sigma}\mathcal{H}(S)_\bullet\xrightarrow{\Sigma}\cdots.$$ We still denote the corresponding infinite suspension by $\Sigma^\infty:\mathcal{H}_\bullet(S)\to \mathcal{SH}^{S^1}(S)$.
\begin{rem}\label[rem]{monoidal of slice localization}
  All the stable $\infty$-categories mentioned so far admit a canonical symmetric monoidal structure such that the functors $\Sigma^\infty_+$ are monoidal with respect to the cartesian monoidal structure on the unpointed categories. This operadic extension arises from recognizing these stable $\infty$-categories as the monoidal inversion of the pointed topological circle $S^1$. Namely, stabilization corresponds to universally inverting $S^1$ with respect to the smash product on the associated pointed spaces ([\cite{robalo2012noncommutative}]):$$Spt^\otimes \simeq Spc_\bullet^{\wedge}[(S^1)^{-1}], \quad \mathcal{SH}^{S^1,\otimes}(S) \simeq \mathcal{H}(S)_\bullet^\wedge[(S^1)^{-1}].$$ 
\end{rem}

\underline{\textbf{Motivic slices}}

The $n$-th Thom connective part of the motivic stable homotopy category, denoted $\mathcal{SH}^{S^1}(S)(n)$, is the localizing tensor-ideal of $\mathcal{SH}^{S^1}(S)$ generated by $T^n$. In other words, it is the essential image of the endofunctor $${T^n}\wedge -:\mathcal{SH}^{S^1}(S)\to \mathcal{SH}^{S^1}(S).$$

In [\cite{voevodsky2002possible}], Voevodsky observes (though in the language of triangulated categories) that the inclusion $$\mathcal{SH}^{S^1}(S)(n)\subset \mathcal{SH}^{S^1}(S),$$ being a colimit-closed embedding of compactly generated stable $\infty$-categories, is a colocalization.
 
Denote these colocalization functors by $$f_n: \mathcal{SH}^{S^1}(S) \to \mathcal{SH}^{S^1}(S)(n).$$ For every $n\geq 0$, the inclusion $$\mathcal{SH}^{S^1}(S)(n+1) \subset \mathcal{SH}^{S^1}(S)(n)$$ induces a canonical natural transformation $f_{n+1} \to f_n$, yielding the slice tower: $$\cdots \to f_{n+1} \to f_n \to \cdots \to f_0 = \mathrm{id}.$$ The $n$-th slice functor $s_n$ is then defined as the cofiber of this natural transformation, yielding the canonical cofiber sequence: $$f_{n+1} \to f_n \to s_n.$$
 
Dually, the $n$th slice localization of the motivic homotopy category is the orthogonal complement of the subcategory $\mathcal{SH}^{S^1}(S)(n)\subset \mathcal{SH}^{S^1}(S)$ denoted by $$\mathcal{SH}^{S^1}(S)/f^{}_{n}\subset \mathcal{SH}^{S^1}(S).$$ So that it is given by the localization $$f_{0/n}: \mathcal{SH}^{S^1}(S)\to \mathcal{SH}^{S^1}(S)/f^{}_{n}.$$ 
Clearly, the tower mentioned above, upon taking cofibers along $f_0=1$, yields another tower:
$$\cdots \to f_{0/n+1}\to f_{0/n}\to \cdots\to f_{0/0}=1$$
\begin{lem}\label[lem]{slices as eilenberg layers}
    The slices $s_n$ are identically given as the fibers of the above tower, i.e., there are fiber sequences $$s_n\to f_{0/n+1}\to f_{0/n}.$$\end{lem}
\begin{proof}
This follows from a snake lemma-type argument. Namely, 
    \begin{flalign*}
        cofib(f_{0/{n+1}}\to f_{0/n})&= cofib\big(cofib(f_{{n+1}}\to f_{0})\to cofib(f_{{n}}\to f_{0}))\\&\simeq cofib\big(cofib(f_{{n+1}}\to f_{n})\to cofib(f_0\to f_0)\big)\\& \simeq cofib(s_n\to 0)\\&\simeq \Sigma s_n
    \end{flalign*}    
\end{proof}

\subsection{Motivic slices through spherical co-localization}
The main goal of this paper is to show that when $S$ is the spectrum of a field, $f_{0/n}$ preserves connectivity for every $n\geq 0$. To do so, the key insight is that these categories can be constructed as further schematic localizations of the motivic stable homotopy category. More precisely, let $\mathrm{S}_\mathbb{A}^n:=\mathbb{A}^{n+1}\setminus 0$ be called the affine $n$-sphere. Then,

\begin{lem}\label[lem]{slice quotient is spherical localization}
    There is a canonical equivalence $$\mathcal{SH}^{S^1}(S)/f_n\simeq L_{\mathrm{S}^{n-1}_\mathbb{A}}\mathcal{SH}^{S^1}(S),$$ where $L_{\mathrm{S}^{n-1}_\mathbb{A}}$ stands for the localization at the set of $\mathrm{S}^{n-1}_\mathbb{A}$ projection maps:
    $$\{\Sigma ^\infty_+{( X\times \mathrm{S}^{n-1}_\mathbb{A}})\to \Sigma ^\infty _+X\}.$$ 
\end{lem}
\begin{proof}
To prove this claim, it suffices to establish another claim that the stable $\infty$-category $\mathcal{SH}^{S^1}(S)(n)$ is equivalent to the localizing $\otimes$-ideal of $\mathcal{SH}^{S^1}(S)$ generated by $(\mathrm{S}^{n-1}_\mathbb{A},1)$, i.e., 
$$\mathcal{SH}^{S^1}(S)(n)=\mathcal{SH}^{S^1}(S)\wedge \mathrm{S}^{n-1}_\mathbb{A}.$$
Recall that, by [\cite{morel19991}, Proposition 3.2.17], there is a cofiber sequence
$$(\mathrm{S}^{n-1}_\mathbb{A},1)\to \mathbb{A}^n\to T^{\wedge n}$$ and hence an equivalence $T^n\simeq \Sigma (\mathrm{S}^{n-1}_\mathbb{A},1)$ in $\mathcal{SH}^{S^1}(S)$. The above claim thus follows from the fact that $\Sigma$ is an equivalence on $\mathcal{SH}^{S^1}(S)$, meaning that $T^{\wedge n}$ and $(\mathrm{S}^{n-1}_\mathbb{A},1)$ generate the same localizing $\otimes$-ideal.  
\end{proof}

\begin{rem}\label[rem]{slice quotient is Lpn}
    Using the same principle, it follows that for every $p$, the localization $L^{p,n}_{sp}$ at the motivic sphere $S^{p,n}$ [\cite{asok2023p}] yields nothing but $\mathcal{SH}^{S^1}(k)/f_n.$ 
\end{rem}
We shall call the localization $L_{\mathrm{S}^{n}_\mathbb{A}}\mathcal{SH}^{S^1}(S)$ the $n$-spherical motivic stable homotopy category, the local objects $n$-spherical motives, and the localization functor the $n$-spherical localization functor. We denote the combined localization $\mathcal{S}pt^{S^1}(S)\to L_{\mathrm{S}^{n}_\mathbb{A}}\mathcal{SH}^{S^1}(S)$ by $L_{\mathbb{A}^1,\mathrm{S}^{n}_\mathbb{A},nis}$. Thus, the above Lemma establishes an equivalence 
\begin{flalign}\label{tate truncation is spherical localization}
    f_{0/n}\simeq {L_{\mathbb{A}^1,\mathrm{S}^{n-1}_\mathbb{A},nis}}_{\mid \mathcal{SH}^{S^1}(S)}.
\end{flalign} So, using \Cref{slices as eilenberg layers}, we obtain fiber sequences $$s_n\to  L_{\mathbb{A}^1,\mathrm{S}^{n}_\mathbb{A},nis}\to  L_{\mathbb{A}^1,\mathrm{S}^{n-1}_\mathbb{A},nis}.$$
Through this lens, $f_{0/n}$'s serve as the Tate version of the "Postnikov truncation," while $s_n$'s serve as Tate "Eilenberg-MacLane
layers".   
\begin{cor}\label[cor]{monoidal structure of tate truncations}
 There is a monoidal structure on $\mathcal{SH}^{S^1}(S)/f_n$ such that the slice truncation functor $$f_{0/n}:\mathcal{SH}^{S^1}(S)\to \mathcal{SH}^{S^1}(S)/f_n$$ can be promoted to a monoidal localization. 
\end{cor}
\begin{proof}
    Since the sets of projection maps \begin{flalign*}
    \{\Sigma ^\infty_+{( X\times \mathrm{S}^{n-1}_\mathbb{A}})\to \Sigma_+ ^\infty X\}_{X\in Sm_S},\\\{\Sigma ^\infty_+{( X\times \mathbb{A}}^1)\to \Sigma ^\infty_+ X\}_{X\in Sm_S}\end{flalign*} are stable under smash products with the compact generators $\Sigma^\infty_+Y$ (for smooth $Y/S$), the localization $L_{\mathbb{A}^1,\mathrm{S}^{n-1}_\mathbb{A},nis}$ is a compatible localization. By [\cite{lurie2007derived}, Proposition 1.3.9], it follows that $L_{\mathbb{A}^1,\mathrm{S}^{n-1}_\mathbb{A},nis}$ is a mononoidal localization. The result thus follows from \Cref{tate truncation is spherical localization}.
\end{proof}
\begin{rem}
    This method is insufficient to analyze the monoidal nature of the remaining functors, namely the $ f_n$'s or $ s_n$'s. As our primary goal is to analyze the connectivity of the slice filtration, we will not require this full monoidal machinery here. For relevant work, we refer the reader to [\cite{pelaez2008multiplicative}].
\end{rem}

\begin{cor}\label[cor]{slice localization is colimit closed}
 Each possible composition in the inclusion $$\mathcal{SH}^{S^1}(k)/f_n\subset \mathcal{SH}^{S^1}(k) \subset \mathcal{S}pt_{nis}^{S^1}(k) $$  are closed under all limits and colimits.
\end{cor}
\begin{proof}
  The limit part is clear by definition.
  
  Since colimits in $\mathcal{S}pt(k)$ are computed sectionwise, schematic localizations of $\mathcal{S}pt(k)$ are closed under all colimits. On the other hand, because the Nisnevich topology is generated by a cd-structure, $\mathcal{S}pt_{nis}(k)\subset\mathcal{S}pt(k)$ is closed under filtered colimits. It follows that the inclusions in the statement are closed under filtered colimits, since both the first category (by \Cref{slice quotient is spherical localization}) and the second are constructed as localizations at schemes. The claim then follows from the standard fact that, in a stable $\infty$-category, every small colimit can be expressed as a filtered colimit of finite coproducts (which are also finite products).
\end{proof}
\subsection{Models for motivic Slices}
In this subsection, we aim to develop models for the slice localization functors $f_{0/n}$. The crucial insight from the previous section is that this process is equivalent to localizing with respect to projections along schemes with rational points. We will explore a more general approach to localizing a stable $\infty$-category at objects with global sections, then apply it to our specific context. But before we do that, let us state some simple observations. 

For a scheme $A \in Sm_S$, let $L_A^{sp}$ denote the localization of $Spt^{S^1}(S)$ with respect to the class of projections $\{X\times_S A \to X\}_{X\in Sm_S}$, and let $L^{sp}_{A,nis}$ denote the further localization at Nisnevich equivalences. Then,
\begin{lem}\label[lem]{LA preserves nis locality}
    There is an equivalence,
    $$L^{sp}_{A,nis}\simeq L_{A}^{sp}\circ L_{nis}^{sp}.$$ In other words, $L^{sp}_A$ preserves Nisnevich locality of $S$-spectra.
\end{lem}
\begin{proof}
By [\cite{hoyois2017six}, Proposition 3.4], there is a formula for localization at the set $\{X\times_S A \to X\}_{X\in Sm_S}$, namely, $$ L_{A}\mal{X}(Y) =\colim_{(\delta^n_{A}\times_S Y)^{op}} \mal{X}(\delta^n_{A}\times_S Y)$$ where $\delta^\bullet_{A}\times Y$ is the sub-diagram of the slice category $({Sm_S})_{/Y}$ spanned by objects of the form $$\delta^n_{A}\times Y:=A^{ n}\times _S Y$$ with structure maps given by the projections to $Y$ and the colimit is computed in the stable $\infty$-category $Spt$. But since (arbitrary) colimits commute with finite limits and Nisnevich locality is checked by a pullback condition, it follows from the above formula that $L_A^{sp}$ preserves Nisnevich local spectra.
\end{proof}

It turns out that, in the unstable case, the formula for $L_A$ as above is sufficient to prove an unstable connectivity result, provided $A$ has a rational point (see [\cite{p1algtop}, Proposition 2.1.10]). However, the same kind of formula is insufficient to produce a stable connectivity result because stable $\pi_0$ fails to preserve colimits in general (see [\cite{p1algtop}, Remark 5.1.4] as well). We shall therefore need a different model for $L^{sp}_A$ to achieve the results of this section. This can be done more generally for an appropriate localization of presentable stable $\infty$-categories, which we state now.

The following is an abstract synthesis of an analogous result from [\cite{ayoub2020P1}].
\begin{con}\label[con]{model for loc in stab}
    Let $\mathscr{C}$ be a presentable, monoidal closed, stable $\infty$-category with a monoidally closed set of generators. Let $A$ be a compact object that is a retract of the monoidal unit $1$, with retraction $$r: A\leftrightarrows1: s.$$

    Choose a set of generators of $\mathscr{C}$ that is closed under monoidal products, contains the unit object and the object $ A$, and call it $\mathcal{G}$. Let $$\mathcal{S}_ A:=\{X\otimes A\xrightarrow[]{X\otimes r} X\}_{X\in \mathcal{G}}.$$ Consider the accessible localization of $\mathscr{C}$ at $\mathcal{S}_ A$. Since the monoidal structure is closed, the localization functor is monoidal.  
    
     Let $ A_\bullet$ be a summand of the above retract (we can do this by stability). Define $\Phi_ A$ as the cofiber of the evaluation $$ A_\bullet\otimes (-)^{ A_\bullet}\to 1_\mathscr{C}.$$ By construction, this yields a canonical cofiber sequence $$ A_\bullet\otimes (-)^{ A_\bullet}\to 1_\mathscr{C}\to \Phi_ A.$$ Let $\Phi^n_A$ denote the $n$-fold composition of $\Phi_A$, and correspondingly, let $$\Phi^\infty_ A:=\dcolim_{m}\Phi^n_ A$$ denote the sectionwise colimit of the tower $$1_\mathscr{C}\to \Phi_ A\to \Phi^2_ A\to \cdots.$$
\end{con}
\begin{prop}\label[prop]{model for stable bousfield localization}
    $\Phi^\infty_ A$ is a model for the monoidal localization with respect to the class $\mathcal{S}_A$.
\end{prop}
\begin{proof}
Since $\mathscr{C}$ is stable, a morphism is an $ A$-equivalence if and only if its (co)fiber is $ A$-contractible, i.e., has contractible $S_ A$-localization. Since $ A\to 1$ is an $ A$-equivalence by definition, the cofiber $ A_\bullet$ is $ A$-contractible, i.e., $L_{ A}{A_\bullet}$ is contractible. But since $L_{ A}$ is a monoidal localization, this implies that $ A_\bullet\otimes (-)^{ A_\bullet}$ is also $ A$-contractible. Thus, by the first line of this proof, we conclude that $1_\mathscr{C}\to \Phi_ A$ and hence $1_\mathscr{C}\to \Phi_ A^\infty$ is an $ A$-equivalence. 

It remains to argue that $\Phi_ A^\infty$ is $ A$-local (i.e., lands in the full subcategory of $ A$-local objects). First, observe that, due to stability, $\esc{X}\in \mathscr{C}$ is $ A$-local iff $\esc{X}^{ A_\bullet}$ is contractible. Now, since $$\Phi_ A\circ\Phi^\infty_ A\simeq \Phi_ A^\infty, $$ the map $$ A_\bullet\otimes (\Phi^\infty_ A)^{ A_\bullet}\to \Phi^\infty_ A$$ is the $0$ map, and hence so is the map $$( A_\bullet\otimes(\Phi^\infty_ A)^{ A_\bullet})^{ A_\bullet}\to (\Phi^\infty_ A)^{ A_\bullet} .$$ By the triangular identity of adjunction, this last map splits. Consequently, $(\Phi^\infty_ A)^{ A_\bullet}$ must be contractible. 
\end{proof}
Now let $\Phi_{\mathbb{A}^1}$ denote the above construction carried out for $A=\mathbb{A}^1$ carried out in $Spt(k)$. We thus obtain the following:
\begin{thm}\label{formula for a1 sn nis localization}
    Over an arbitrary base, there are equivalences of localization functors $$L_{mot}^{sp}\simeq {( L_{nis}^{sp}\Phi_{\mathbb{A}^1})}^\infty L^{sp}_{nis}\simeq\Phi_{\mathbb{A}^1}^\infty L_{nis}^{sp} ,$$  
    \begin{flalign*}
       f_{0/n+1}L^{sp}_{mot}\simeq L^{sp}_{\mathbb{A}^1,\mathrm{S}^{n}_\mathbb{A}, nis}&\simeq \dcolim_{m}\big(\Phi_{\mathrm{S}^{n}_\mathbb{A}}^\infty L_{mot}^{sp}\big )^m\\&\simeq \big (\dcolim_{m}\big(\Phi_{\mathrm{S}^{n}_\mathbb{A}}^\infty \Phi_{\mathbb{A}^1}^\infty\big )^m\big)\circ L_{nis}^{sp}\\
       &\simeq \dcolim_{m}\big( L_{mot}^{sp}\Phi_{\mathrm{S}^{n}_\mathbb{A}}\big )^m\circ L^{sp}_{mot}
    \end{flalign*}
\end{thm}
\begin{proof}
For the first equivalence in the first formula, we note that $L_{nis}^{sp}\Phi_{\mathbb{A}^1}$ is the candidate for $\Phi$ in $Spt^{S^1}_{nis}(S)$ for the object $\mathbb{A}^1$. Since colimits in $Spt^{S^1}_{nis}(S)$ are computed in $Spt^{S^1}(S)$, ${( L_{nis}^{sp}\Phi_{\mathbb{A}^1})}^\infty$ is nothing but $\Phi^\infty_A$ in $Spt^{S^1}_{nis}(S)$ for $A=\Sigma^\infty_+\mathbb{A}^1$.

For the second formula in the first equivalence, note that by \Cref{model for stable bousfield localization} $\Phi^\infty_{\mathbb{A}^1}$ is a model for $\mathbb{A}^1$-localization in $Spt^{S^1}(S)$. So this is just a consequence of \Cref{LA preserves nis locality}.
    
We now turn to the formula for $f_{0/n+1}L^{sp}_{mot}$. The relative slices $f_{0/n+1}$ coincide with $(\mathrm{S}^{n}_\mathbb{A})$-localization, as stated in \Cref{tate truncation is spherical localization}. The formula $\dcolim_{m}\big(\Phi_{\mathrm{S}^{n}_\mathbb{A}}^\infty L_{mot}^{sp}\big )^m$ is standard: use the fact that $\mathcal{SH}^{S^1}(k),L_{\mathrm{S}^n}Spt_{}(S)$ are closed under (filtered) colimits in $Spt^{S^1}(S)$, while saturation classes of local equivalences are closed under (all) colimits. 

We need to argue that the third term provide simultaneous $(\mathrm{S}^{n}_\mathbb{A})$-localization and $\mathbb{A}^1$-localization functors in $\mathcal{S}pt_{nis}(S)$. In this, the term $\dcolim_{m}\big(\Phi_{\mathrm{S}^{n}_\mathbb{A}}^\infty \Phi_{\mathbb{A}^1}^\infty\big )^m$ is clearly a formula for simultenious  $(\mathrm{S}^{n}_\mathbb{A})$-localization and $\mathbb{A}^1$-localization in $Spt^{S^1}(S)$. The claim follows again from \Cref{LA preserves nis locality} (or more generally from the method of its proof).

The case of the last formula is similar to the argument of the first paragraph of this proof: $\dcolim_{m}\big( L_{mot}^{sp}\Phi_{\mathrm{S}^{n}_\mathbb{A}}\big )^m$ is nothing but  $\Phi^\infty_A$ in $\mathcal{SH}^{S^1}(S)$ for $A=\Sigma^\infty_+\mathrm{S}^n_\mathbb{A}$.
\end{proof}

\subsection{Slice connectivity}
In this section, we shall finally prove the connectivity property for $f_{0/n}$ over fields. Due to \Cref{tate truncation is spherical localization}, this is equivalent to studying the connectivity property of the motivic spherical nullification functor $L_{{\mathbb{A}^1}, \mathrm{S}^n_\mathbb{A}, nis}^{sp}$. Thus, the aim of this section is to prove the following version of the stable connectivity theorem.
\begin{recall*}{slice connectivity}
When the base is the spectrum of a field, the localization $L_{{\mathbb{A}^1}, \mathrm{S}^n_\mathbb{A}, nis}^{sp}$ preserves Nisnevich-local connectivity. It follows that $f_{0/n+1}$ preserves motivic connectivity. 
\end{recall*}
To do so, let us recall the following terminology from [\cite{ayoub2020P1}], where $k$ is an arbitrary field:
\begin{defn}
 A presheaf of spectra $\mal{A}\in Spt^{S^1}(k)$ is said to be:
 \begin{enumerate}
     \item $m$-pre-connected if, for all essentially smooth schemes $\mathcal{O}/k$, the spectra $\mal{A}(\mathcal{O})$ are $(m-dim\mathcal{O})$-connected.
     \item weakly $m$-connected (also called generically connected in \textup{[\cite{ayoub2020P1}]}) if for every $K\in \mathcal{F}_k$ the spectra $\mal{A}(K)$ is $m$-connected.
 \end{enumerate}
\end{defn}
For $X\in Sm_k$ with a rational point $x$, let us denote by $\mathbb{X}:=(X,x)$ the representable presheaf pointed by the global section $x$.
\begin{lem}\label[lem]{function space of pre-connected is -dim pre-conneceted}
     Suppose $X$ has Krull dimension at least $1$. Then for any $m$-pre-connected presheaf of spectra $\mal{A}$, the associated function spectra $\mal{A}^{\Sigma^\infty\mathbb{X}}$ is $(m-dimX)$-pre-connected.
\end{lem}
\begin{proof}
First, note that the spectrum $$\mal{A}^{\Sigma_+^\infty X}$$ is $(m-dimX)$-preconnected, since the product of $X$ with any essentially smooth scheme of dimension $d$ is essentially smooth of dimension $d+dimX$. Moreover, the spectrum $\mal{A}$, being $m$-preconnected, is automatically $m-dimX+1$ ($\leq m$)-preconnected. Now, the cofiber sequence $$\Sigma^\infty_+ Speck\xrightarrow{x_+}\Sigma^\infty _+ X\to \Sigma^\infty \mathbb{X}$$ yields the fiber sequence $$\mal{A}^{\Sigma^\infty \mathbb{X}}\to \mal{A}^{\Sigma_+^\infty X}\to \mal{A}.$$ From this, it follows that $\mal{A}^{\Sigma^\infty\mathbb{X}}$ is $(m-dimX)$-preconnected (for example, via a long exact sequence of stable homotopy groups). 
\end{proof}
\begin{lem}\label[lem]{presheaf phiA preserve pre-connectivity}
    For a scheme $X$ of dimension $1$ with a rational point, the associated spectral localization $L_{X}^{sp}$ preserves $m$-pre-connectivity. 
\end{lem}
\begin{proof}
Note that the existence of a rational section implies that we have a model of $X$ localization given by $L_X\simeq \Phi_X^\infty$ (\Cref{model for stable bousfield localization}). Since connectivity is stable under filtered colimits, it suffices to show that $\Phi_X$ takes $m$-preconnected spectra to $m$-preconnected spectra. 

So let $\mal{A}$ be a $m$-preconnected $k$-spectra. From the above lemma we know that $\mal{A}^{\Sigma^\infty \mathbb{X}}$ is $(m-1)$ pre-connected. Since $\Sigma^\infty \mathbb{X}$ is sectionwise $-1$ connected it follows at once that $$\Sigma^\infty\mathbb{X}\wedge \mal{A}^{\Sigma^\infty \mathbb{X}}$$ is $(m-1)-1+1=m-1$ pre-connected. Because $\Phi_X$ is defined as the cofiber of $$\Sigma^\infty\mathbb{X}\wedge \mal{A}^{\Sigma^\infty \mathbb{X}}\to \mal{A}$$ the claim follows again by the long exact sequence of presheaves of homotopy groups. 
\end{proof}

\begin{lem}\label[lem]{nis conn to pre conn}
   For a presheaf of spectra $\mal{A}$, the nisnevich localization $L_{nis}^{sp}\mal{A}$ is $m$-pre-connected under either of the following hypotheses:
   \begin{enumerate}
       \item   $\mal{A}$ is nisnevich $m$-connected
       \item  $\mal{A}$ is $m$ pre-connected.
   \end{enumerate}
\end{lem}
\begin{proof}
 It suffices to consider $m=-1$. We must show that if $X$ is an essentially smooth scheme, then $$\pi_l^s(L_{nis}^{sp}\mal{A})(X)=0\text{ for all }l<-dim X.$$ Recall that there is a strongly convergent nisnevich descent spectral sequence associated to $L_{nis}^{sp}\mal{A}$:
 $$H_{nis}^a(X,\pi_{-b}^{sp,nis}L_{nis}^{sp}\mal{A})\implies \pi_{-b-a}^{sp}L_{nis}^{sp}\mal{A}$$
    from which it suffices to show that $$H_{nis}^a(X,\pi_{-b}^{sp,nis}L_{nis}^{sp}\mal{A})\simeq H_{nis}^a(X,\pi_{-b}^{sp,nis}\mal{A}) \cong H_{nis}^a(X,\pi_{-b}^{sp}\mal{A})=0$$ for all $l=-b-a< -dim X$ i.e. $-b+dimX<a$. Since the Nisnevich cohomological dimension of $X$ is bounded by $dimX$ [\cite{bachmann2024strongly}, Theorem 1.14], the demand is met for $b\leq 0$. 
    
    On the other hand, if $b> 0$ then $-b+dimX<a$ implies that $a<dimX$. In this situation, the following cases can occur: 
    \begin{enumerate}
        \item \textbf{Case 1:} If $\mal{A}$ is Nisnevich $(-1)$-connected, then for all $b>0$ we have $\pi_{-b}^{sp,nis}\mal{A}=0$. Thus, the requirement is satisfied.
        \item {\textbf{Case 2}}: If $\mal{A}$ is $(-1)$-preconnected, let $F=\pi_{-b}^{sp}\mal{A}$. Let $Y/X$ be etale, and let $y\in Y$ be a point such that $cod_Y(Y)<b$, so that $dim(Y_y)=cod_Y(y)<b$. Because $-b<-dim(Y_y)$, we know from the $(-1)$-preconnectivity of $\mal{A}$ that $F(Y_y)=0$. As $X$ is noetherian, by [\cite{ayoub2020P1}, Lemma 4.9] it follows that $$H_{nis}^a(X,F)=0\text{ for all }a>dimX-b $$ and we are done in this case as well, since by definition $$H_{nis}^a(X,\pi_{-b}^{sp,nis}\mal{A})\cong H_{nis}^a(X,F).$$
    \end{enumerate}
\end{proof}
\begin{cor}\label[cor]{nis to pre to weak}
    Let $\mal{A}$ be a pre-sheaf of spectras and consider the statements:
    \begin{enumerate}
        \item $\mal{A}$ is nisnevich $m$ connected
        \item $\mal{A} $ is $m$ pre connected.
        \item $\mal{A}$ is weakly $m$-connected.
    \end{enumerate}
    Then (2)$\implies$(3). If $\mal{A}$ is Nisnevich local, then (1)$\implies$(2).
\end{cor}
\begin{proof}
     (2)$\implies$(3) follows from the definition. Since Nisnevich localization preserves Nisnevich connectedness (trivially, by definition), (1)$\implies$(2) follows from the lemma above \Cref{nis conn to pre conn}.
\end{proof}

\begin{cor}\label[cor]{spectra A nis localization preserves pre connectivity}
  For a scheme $A$ of dimension $1$ with a rational point, the localization $L_{A}^{nis}$ preserves $m$-pre-connectivity. Consequently, $L_{A}^{nis}$ maps nisnevich $m$-connected spectra to $m$-pre-connected spectra. In particular, $L_{mot}^{sp}$ preserves $m$-pre-connectivity and maps nisnevich $m$-connected spectra to $m$-pre-connected spectra. 
\end{cor}
\begin{proof}
    Because $L_{nis}^{sp} $ (by \Cref{nis conn to pre conn}) and $L_A$ (by \Cref{presheaf phiA preserve pre-connectivity}) both preserve $m$ pre-connectivity, the claim follows from the formulae $$L^{nis}_A\simeq L_AL_{nis}^{sp}.$$

    Since $L^{nis}_A\simeq L^{nis}_AL_{nis}^{sp}$, the other claim follows from the first one and \Cref{nis conn to pre conn}.
\end{proof}
\begin{rem}
    This generalizes Morel's $\mathbb{A}^1$-connectivity theorem and is used to establish the $\mathbb{P}^1$-local stable connectivity theorem in another work of the author [\cite{p1algtop}].
\end{rem}
 \begin{lem}[Pre-connectivity theorem of slices]\label[lem]{preservation of pre n conn for spherical loca}
Over any field the spherical motivic localization $L_{{\mathbb{A}^1}, \mathrm{S}^n_\mathbb{A}, nis}^{sp}$ takes $m$-pre-connected spectras to $m$-pre-connected spectras.
\end{lem}
\begin{proof}
Since pre connectivity is stable under colimits, under theorem \Cref{formula for a1 sn nis localization}, specifically that $$L_{{\mathbb{A}^1}, \mathrm{S}^n_\mathbb{A}, nis}^{sp} \simeq \colim_{p}(L_{mot}^{sp}\Phi_{\mathrm{S}^m_{\mathbb{A}}}\big )^p\circ L^{sp}_{mot},$$ it suffices to show that $L_{mot}^{sp}\Phi_{\mathrm{S}^m_{\mathbb{A}}}$
takes $m$ pre-connected motivic spectra $\mal{A}$ to $m$ pre connected spcetras (since $L^{sp}_{mot}$ preserves connectivity by \Cref{spectra A nis localization preserves pre connectivity}). We have to show that the fiber defining $$L^{sp}_{mot}\mal{A}\to L_{mot}^{sp}\Phi_{\mathrm{S}^m_{\mathbb{A}}}\mal{A}$$ is $(m-1)$ pre-connected. Now the fiber can be identified as $$L_{mot}^{sp}\big(L_{mot}^{sp}\Sigma^\infty\mathrm{S}^n_\mathbb{A}\wedge L^{sp}_{mot}\mal{A}^{\mathrm{S}^n_\mathbb{A}}\big).$$ From \cref{function space of pre-connected is -dim pre-conneceted}, we know that $\mal{A}^{\mathrm{S}^n_\mathbb{A}}$ is $(m-n-1)$-pre-connected, and hence so is $L^{sp}_{mot}\mal{A}^{\mathrm{S}^n_\mathbb{A}}$. On the other hand, since $$L_{mot}\mathrm{S}^{n}_\mathbb{A}\simeq \Sigma^{n}L_{mot}\mathbb{G}^{n+1}$$ (which is is $(n-1)$ connected), by stable motivic connectivity (\Cref{spectra A nis localization preserves pre connectivity}) we have that $L_{mot}^{sp}\Sigma^\infty\mathrm{S}^n_\mathbb{A}$ is $(n-1)$-connected.

Since smash product raises connectivity by $1$, we find that $L_{mot}^{sp}\Sigma^\infty\mathrm{S}^n_\mathbb{A}\wedge L^{sp}_{mot}\mal{A}^{\mathrm{S}^n_\mathbb{A}}$ is $$(m-n-1)+(n-1)+1=(m-1)$$-connected. So the $(m-1)$ pre-connectivity of the fiber follows from the stable motivic pre-connectivity theorem, \Cref{spectra A nis localization preserves pre connectivity}.
\end{proof}
\begin{cor}\label[cor]{slice connectivity}
    Over any field, $L_{{\mathbb{A}^1}, \mathrm{S}^n_\mathbb{A}, nis}^{sp}$ and hence $f_{0/n+1}$ preserves connectivity. 
\end{cor}
\begin{proof}
  Since $$L_{{\mathbb{A}^1}, \mathrm{S}^n_\mathbb{A}, nis}^{sp}\simeq L_{{\mathbb{A}^1}, \mathrm{S}^n_\mathbb{A}, nis}^{sp}L_{nis}^{sp}$$ and $L_{nis}^{sp}$ takes nisnevich $n$-connected spectras to $m$-pre-connected spectras (\Cref{nis conn to pre conn}(1)), it follows from \Cref{preservation of pre n conn for spherical loca} that if $\mal{A} $ is nisnevich $n$-connected, then $L_{{\mathbb{A}^1}, \mathrm{S}^n_\mathbb{A}, nis}^{sp} \mal{A}$ is $m$-pre-connected. By \Cref{nis to pre to weak}, we deduce that $L_{{\mathbb{A}^1}, \mathrm{S}^n_\mathbb{A}, nis}^{sp} \mal{A}$ is weakly (/generically) $n$-connected. Since $L_{{\mathbb{A}^1}, \mathrm{S}^n_\mathbb{A}, nis}^{sp} \mal{A}$ is motivic local by definition/construction and connectivity of motivic spectras is identical to generic connectivity [\cite{morel2005stable}, Lemma 6.1.6], we immediately obtain nisnevich $n$-connectivity of $L_{{\mathbb{A}^1}, \mathrm{S}^n_\mathbb{A}, nis}^{sp} \mal{A}$ as desired. 
\end{proof}
\begin{cor}\label[cor]{slices takes connected to connective}
 Let $k $ be a field. Then, for every $n\geq 0$
 \begin{enumerate}
     \item $s_n:\mathcal{SH}^{S^1}(k)\to \mathcal{SH}^{S^1}(k)$ takes connected spectras to connective spectras. 
     \item $f_n:\mathcal{SH}^{S^1}(k)\to \mathcal{SH}^{S^1}(k)$ takes connected spectras to connective spectras. 
 \end{enumerate}
\end{cor}
\begin{proof}
    (1) follows immediately from \Cref{slices as eilenberg layers} and the slice connectivity theorem stated above \Cref{slice connectivity}. 

    (2) follows similarly from \Cref{slices as eilenberg layers} and the fiber sequence defining $f_{0/n}$, namely
    $$f_n\to 1\to f_{0/n}.$$
\end{proof}

\begin{cor}\label[cor]{zero slice connectivity}
    Over a field $k $ the zeroth slice $s_0:\mathcal{SH}^{S^1}(k)\to \mathcal{SH}^{S^1}(k)$ preserves connectivity.
\end{cor}
\begin{proof}
 This is immediate from the fact that $s_0=f_{0/1}$, and from connectivity of $f_{0/1}$ (\Cref{slice connectivity}).
\end{proof}

 \underline{\textbf{Known proofs of slice connectivity in the literature}}

\begin{rem}\label[rem]{levine}
 Over infinite perfect fields, one can also use Levine's model of slices.

First, recall the following coneavue constructions from [\cite{levine2008homotopy}] for a presheaf of spectra $\mal{E}:$\begin{flalign}
&{X\mapsto \esc{E}^{(n)}(X):=\Big | \colim_{ W\in S^{(n)}(X,\bullet)}\esc{E}^{ W}(X\times \Delta ^\bullet) \Big |}\\
  &{X\mapsto\esc{E}^{(0/n)}(X):=\Big | \colim_{Z\in S^{(0)}(X,\bullet), W\in S^{(n)}(X,\bullet)}\esc{E}^{Z\setminus W}(X\times \Delta ^\bullet\setminus W) \Big |}
  \end{flalign}
  \end{rem} 

\begin{cor}\label[cor]{tate truncation coniveau model}
  Let $k$ be an infinite perfect field. We then have the following exact models of tate truncation for presheaves of spectras on $Sm_k$:
  $$f_{0/n+1}\simeq (L_{mot}^{sp})^{^{(0/n+1)}}\simeq \dcolim_{m}\big(\Phi_{\mathrm{S}^n_\mathbb{A}} \Phi_{\mathbb{A}^1} L_{nis}^{sp}\big )^m.$$
\end{cor}
\begin{proof}
   Using \Cref{slice connectivity}, it is a restatement of [\cite{levine2008homotopy}, Theorem 7.1.1]. 
\end{proof}
It is possible to obtain a sectionwise connectivity result using the above models of slices over infinite perfect fields (although this model appears insufficient to establish the Nisnevich connectivity result targeted in this paper).  
\begin{proof}
(Sectionwise slice connectivity using Levine's model.)
 First, let $\esc{E}$ be a sectionwise connected motivic spectrum. We want to show that $\esc{E}^{0/n+1}$ is also sectionwise connected. By construction, there is a cofiber sequence of motivic spectra $$\esc{E}^{(n)}\to \esc{E}\to \esc{E}^{(0/n+1)}.$$ Thus it suffices to argue that $\esc{E}^{(n)}$ is connective. But this is simply the geometric realization of the simplicial motivic spectra $r\mapsto \esc{E}^{(n)}(-,r)$. By the Bousfield-Kan spectral sequence, it is thus enough to prove that each $\esc{E}^{(n)}(-,r)$ is connective. On sections over $X$, this last one is given by a filtered colimit over $Z\in S^{(n+1)}(X,r)$ of $$\esc{E}^{Z}(\Delta_X^r):=fiber \big(\esc{E}(\Delta^r_X)\to \esc{E}(\Delta_X^r\setminus Z\big ).$$ Since $\esc{E}$ is sectionwise connected, it is obvious (using long exact sequences) that the fibers $\esc{E}^{Z}(\Delta_X^r)$ are connective, whence so is their colimit $\esc{E}^{(n)}(X,r)$.
\end{proof}

\begin{rem}\label[rem]{bachmann gm conservativity}
The consequences of this machine, namely the validity of the slice conjectures and thus the conservativity of $\omega_{\mathbb{G}}:\mathcal{SH}^{S^1}(k)(n+1)\to \mathcal{SH}^{S^1}(k)(n)$ [\cite{MR4173925}, Lemma 6.1], have, however, been used in \textit{loc.cit} to produce the required connectivity result, in fact a stronger result (assuming the base is perfect) in Corollary 6.2 (2). Indeed, in [\cite{MR4173925}], the author shows that, over perfect fields, $f_n$ is right $t$-exact for the homotopy $t$-structure on $\mathcal{SH}^{S^1}(k)$ (use Corollary 6.2 (1,2) and that $f_n=i_nr_n$). From this, our result on the right exactness of $f_{0/n}$ (being a cofiber) follows immediately. However, this approach relies on the validity of the slice conjecture over infinite perfect fields (see footnote 10 in loc. cit.), a substantial result established in [\cite{levine2008homotopy}] (see [\cite{bachmann2019voevodsky}] for a streamlined proof of a different form of the conjecture using the motivic recognition principle). The alternative method developed in this paper, discussed thus far for $f_{0/n}$, proceeds directly and is entirely independent of the slice conjectures. We shall come back to the connectivity property of $f_n$'s on $\mathcal{SH}^{eff}$ later in \S5 through the lens of framed correspondences. \end{rem}
\subsection{Complete accessible t-structures}

Before we write down what happens to the $t$-structure of $\mathcal{SH}^{S^1}$ after we perform slice localizations, we shall take a detour and write down a general theory.

Let $\mathscr{C}$ be a stable, presentable $\infty$-category equipped with an accessible $t$-structure $(\mathscr{C}_{\geq 0},\mathscr{C}_{\leq 0})$, and let $L:\mathscr{C}\to \mathscr{C}$ be an exact localization functor with saturation class $S_L$. We first set up the following standard result on localization of stable $\infty$-categories with $t$-structures:
\begin{prop}[Localization of t structures] \label{localization of t structures}
If the localization $L: \mathscr{C}\to \mathscr{C}$ is right $t$-exact, then $$(L\mathscr{C}\bigcap \mathscr{C}_{\geq 0},L\mathscr{C}\bigcap \mathscr{C}_{\leq 0})$$ is a $t$-structure on $L\mathscr{C}$. This $t$-structure is accessible. If the $t$-structure on $\mathscr{C}$ is left complete, so is the induced $t$-structure on $L\mathscr{C}$. If the $t$-structure on $\mathscr{C}$ is right complete and the inclusion $L\mathscr{C}\subset \mathscr{C}$ is stable under telescopic colimits, then the induced $t$-structure is right complete. The heart of this $t$-structure is $\mathscr{C}^\heartsuit\cap L\mathscr{C}$.
\end{prop}
\begin{proof}
  We verify the conditions of [\cite{lurie2017higher}, Definition 1.2.1.1.].  The orthogonality (condition (1) in loc.cit.) and $(\Omega, \Sigma)$ stability (condition (2) in loc.cit.) are clear.  It remains to establish that for every $X\in L\mathcal{C}$ there is a triangle $$X'\to LX\to X''$$ in $L\mathcal{C}$ such that $X'\in L\mathscr{C}\bigcap \mathscr{C}_{\geq 0}$ and $X''\in L\mathscr{C}\bigcap \mathscr{C}_{\leq 1}$. First consider the $t$-triangle of $X$ in $\mathcal{C}$:
$$X_{\geq 0}\to X\to X_{\leq 0}$$
and apply $L$ to get the triangle $$L(X_{\geq 0})\to LX\to L(X_{<0}).$$ We have the morphism of triangles
    \[
    \xymatrix{
       X_{\geq 0}\ar[r] \ar@<0.5ex>[d]&X\ar[r]\ar[d]^{\simeq }& X_{\leq 0}\ar[d]\\
    L(X_{\geq 0})\ar@<0.5ex>@{-->}[u]\ar[r]&LX\ar[r]& L(X_{< 0})
    }
    \]
where the first vertical dashed arrow pointing upward is induced by the fact that $L(X_{\geq 0}) \in \mathscr{C}_{\geq 0}$, and hence the composite $$L(X_{\geq 0})\to LX\underset{\simeq }{\to } X$$ must factor through $X_{\geq 0}$. This verifies that $X_{\geq 0}$ is a retract of the local object $L(X_{\geq 0})$ and is therefore local itself, i.e., $L(X_{\geq 0})\in L\mathscr{C}\bigcap \mathscr{C}_{\geq 0}$. It follows immediately from stability that $L(X_{< 0})\in L\mathcal{C}$. So take $X'=L(X_{\geq 0})$ and $X''=L(X_{<0})$.

To verify accessibility, one checks that $L\mathscr{C}\bigcap \mathscr{C}_{\geq 0}$ is the essential image of the restriction of the accessible localization $L$ to the accessible subcategory $\mathscr{C}_{\geq 0}$ (this makes sense by right $t$-exactness of $L$).

Left completeness is immediate by the closedness of limits. The claim about right-completeness follows immediately by definition from the given condition. 
\end{proof}
We will call the induced $t$-structure the localized $t$-structure and denote this by $(L\mathscr{C}_{\geq 0},L\mathscr{C}_{\leq 0})$. It is clear by construction that the inclusion $L\mathscr{C}\subset \mathscr{C}$ is $t$-exact. 
\vspace{.5cm}
\begin{obs}\label[obs]{further t localization}
Suppose $\mathcal{T}$ is an $\infty$-topos. Then $ \mathrm{Stab}(\mathcal{T})\simeq Shv_{Sp}(\mathcal{T})$ has a canonical notion of a homotopy connectivity that gives rise to a $t$-structure. An exact localization of $\mathrm{Stab}(\mathcal{T})$ preserves connectivity if and only if $L: \mathrm{Stab}(\mathcal{T})\to \mathrm{Stab}(\mathcal{T})$ is right $t$-exact for this $t$-structure. If $$R: L\mathrm{Stab}(\mathcal{T})\to L\mathrm{Stab}(\mathcal{T})$$ is another localization, then the same applies to $R$.
\end{obs}
\begin{rem}[Motivic Postnikov $t$-structure]\label[rem]{Motivic t structure}
When $\mathcal{T}$ is the Nisnevich $\infty$-topos $\mathcal{P}_{nis}(k)$, the resulting canonical $t$-structure on $\mathrm{Stab}(\mathcal{P}_{nis}(k))=\mathcal{S}pt_{nis}^{S^1}(k)$ is given by 
$$\big( \tau_{\geq 0}^{nis}\mathcal{S}pt^{S^1}_{nis}(k),\tau_{< 0}^{nis}\mathcal{S}pt^{S^1}_{nis}(k)\big)$$
If one applies the above observation to the motivic homotopy category, one arrives at a $t$-structure on $\mathcal{SH}^{S^1}(S)$, provided the motivic stable connectivity theorem holds. When $S$ is the spectrum of a field, the motivic stable connectivity theorem has been verified in [\cite{morel2005stable}, Lemma 6.2.6] (this also follows from our \Cref{spectra A nis localization preserves pre connectivity}, \Cref{nis to pre to weak} in the presence of [\cite{morel2005stable}, Lemma 6.1.6]). The resulting $t$-structure is called the motivic homotopy $t$-structure [\cite{morel2005stable}, Lemma 6.2.11] and is thus given by:
$$\big(\mathcal{SH}^{S^1}(k)\bigcap \tau_{\geq 0}^{nis}\mathcal{S}pt^{S^1}_{nis}(k), \mathcal{SH}^{S^1}(k)\bigcap \tau_{< 0}^{nis}\mathcal{S}pt^{S^1}_{nis}(k)\big).$$
The identification of the heart is interesting in this case. From the above theorem, it is $$EM(\mathcal{S}pt^{S^1}_{nis}(S))\bigcap \mathcal{SH}^{S^1}(S).$$ Of course, $\pi_0^{sp, nis}$ induces an equivalence $EM(\mathcal{S}pt^{S^1}_{nis}(S))\simeq Ab^{nis}_S$. Since the presheaves $\pi^{sp,nis}_{i}$ are precisely the nisnevich cohomology group sheaves for Eilenberg-MacLane spectra of sheaves of groups, the result is precisely the full subcategory of nisnevich sheaves of abelian groups all of whose nisnevich cohomology groups are $\mathbb{A}^1$-invariant. These are what Morel calls strictly $\mathbb{A}^1$-invariant sheaves of abelian groups.
\end{rem}
\begin{rem}
   Similar to the $S^1$-Postnikov truncation, the $n$-Tate covering and truncation do define a $t$-structure for each $n$:$$\big( f_n\mathcal{SH}^{S^1}(k), f_{<n}\mathcal{SH}^{S^1}(k) \big).$$However, this $t$-structure is degenerate: because both subcategories are stable, the resulting heart is trivial. Furthermore, unlike the topological case, these truncations cannot be assembled into a single $t$-structure generated by topological shifts. Accommodating this varying filtration requires generalizing the notion of a $t$-structure to one generated by the action of a monoidal object (e.g., Tate twists), which does not exist yet.
\end{rem}
The slice variant of the notion of strictly $\mathbb{A}^1$-invariant sheaves of abelian groups is as follows:
\begin{defn}
    A presheaf of abelian groups $G\in Sm_S$ is said to be strictly $n$-spherical if, for all $i$, $H^i_{nis}(-,G)$ is $\mathrm{S}^{n}_\mathbb{A}$-local.    
    \end{defn}
    \begin{lem}\label[lem]{strictly n spherical has trivial n+1 contraction}
        Over a perfect field, a strictly $\mathbb{A}^1$-invariant Nisnevich sheaf $G$ of abelian groups is strictly $n$-spherical iff $G_{-n-1}=0$.
    \end{lem}
    \begin{proof}
    Clearly, such a $G$ is strictly $n$-spherical iff for all $i$ the space $\mathrm{B}^{i}_{nis}G$ is an $\mathrm{S}^{n}_\mathbb{A}$-local motivic space. But $\mathrm{S}^{n}_\mathbb{A}\simeq \Sigma^{n}\mathbb{G}^{n+1}$, so that for $i<n$ the space $\mathrm{B}^{i}_{nis}G$ is always $n$-spherical. Therefore, $G$ is strictly $n$-spherical iff for all $i\geq n$ the space $\mathrm{B}^{i}_{nis}G$ is an $\mathrm{S}^{n}_\mathbb{A}$-motivic space. Now, if $i\geq n$, then $\mathrm{B}^i_{nis}G$ is $n$-spherical iff $$0=(\mathrm{B}^i_{nis}G)^{\mathrm{S}^{n}_\mathbb{A}}\simeq \Omega_{\mathbb{G}}^{n+1}\mathrm{B}^{i-n}_{nis}G\simeq \mathrm{B}^{i-n}_{nis}(G_{-n-1})$$ (see [\cite{bachmann2024strongly}, Lemma 4.2] for the last equivalences). This is clearly equivalent to the $i=n$ case, namely $G_{-n-1}=0$.        
\end{proof}
\begin{rem}\label[rem]{n spherical is lpn null}
    Using the same argument (or [\cite{asok2023p}, Corollary 3.1.21]), it follows that, over a perfect field, the conditions of the above Lemma are equivalent to $\mathrm{B}_{nis}^iG$ being $L^{p,n+1}$-null for all $i\geq 0$. 
\end{rem}

\begin{thm}\label{t structure on n spherical motivic spectras}
    Let $k$ be a field. The $n$-th Tate-truncated category $\mathcal{SH}^{S^1}(k)/f_{n}$ admits an accessible, complete t-structure such that each of the functors (and hence the composite) in the chain of inclusions below $$\mathcal{SH}^{S^1}(k)/f_n\subset \mathcal{SH}^{S^1}(k) \subset \mathcal{S}pt_{nis}^{S^1}(k) $$ is $t$-exact. The heart of this $t$-structure is the category of strictly $(n-1)$-spherical, strictly $\mathbb{A}^1$-invariant sheaves of abelian groups. When $k$ is a perfect field, this heart is equivalent to the kernel of the functor $$(-)_{-{n}}: Ab_k^{\mathbb{A}^1}\to Ab_k^{\mathbb{A}^1}.$$
  \end{thm}
  \begin{proof}
In the presence of \Cref{slice connectivity}, this follows from the theorem on localization of $t$-structures, \Cref{localization of t structures}. Since $$f_{0/n}: \mathcal{S}pt_{nis}^{S^1}(k)\to  \mathcal{S}pt_{nis}^{S^1}(k), \quad f_{0/n}: \mathcal{SH}^{S^1}(k)\to \mathcal{SH}^{S^1}(k)$$ preserve connectivity, they are right $t$-exact by observation \Cref{further t localization}. Whence, Proposition \Cref{localization of t structures} applies, yielding the following accessible, left-complete $t$-structure on $$(\mathcal{SH}_{S^1}(k)/f_n\cap \mathcal{SH}^{nis}_{S^1\geq 0}, \mathcal{SH}_{S^1}(k)/f_n\cap \mathcal{SH}^{nis}_{S^1\leq 0}).$$ Since the inclusions are closed under (filtered) colimits (\Cref{slice localization is colimit closed}), right completeness follows from the same proposition. For the statement regarding the inclusion $\mathcal{SH}^{S^1}/{f_n}\subset \mathcal{SH}^{S^1}$, one must first note the contents of \Cref{Motivic t structure}, the stable $\mathbb{A}^1$ connectivity theorem over fields [\cite{morel2005stable}, Theorem 6.1.8], and \Cref{further t localization}. 

Regarding the comment about the heart, the discussion is the same as the one above. We can identify the heart as the full subcategory of Nisnevich Eilenberg-MacLane spectra, all of whose presheaf stable homotopy groups are $n$-birational. These are clearly Nisnevich sheaves of abelian groups whose Nisnevich cohomology groups are $n$-birational. This shall be the category of sheaves of abelian groups, all of whose Nisnevich cohomology groups are $\mathrm{S}^n_\mathbb{A}$-local and $\mathbb{A}^1$-invariant. This is precisely the definition of strictly $\mathbb{A}^1$-invariant, strictly $n$-spherical Nisnevich sheaves of abelian groups. The comment over perfect fields follows from the lemma just above \Cref{strictly n spherical has trivial n+1 contraction}.
\end{proof}
  We shall denote this heart by ${Ab^{\mathbb{A}^1}_k}{/f_n}$.
\begin{cor}\label[cor]{slice t-exact}
     $f_{0/n}$ is right $t$-exact for the homotopy $t$-structure.
  \end{cor}
 
\begin{thm}[excerpted from [\cite{morel2005stable}, Lemma 6.2.7\text{]}]\label{cover judgment em judgment to local}  
    Under the conditions of \Cref{localization of t structures}, for an object $X\in\mathscr{C}$, the following conditions are equivalent:
    \begin{enumerate}
        \item $X$ is $L$-local.
        \item All connective covers $X_{\geq m}$ are $L$-local for all $m$.
        \item The "Eilenberg-MacLane" layers, i.e., the $t$-homotopy groups $\pi_m^tX$ are $L$-local for all $m$.
        \item The $t$-homotopy groups $\pi_m^tX$ satisfy $$\pi_m^tX\in (L\mathscr{C})^\heartsuit = L\mathscr{C}\cap\mathscr{C}^\heartsuit.$$
    \end{enumerate}
  \end{thm}
  \begin{proof}
(1)$\implies$(2) is embedded in the proof of \Cref{localization of t structures}. Since the localized category is closed under limits, (2)$\implies$(3) is obvious. 
  
The only nontrivial part is (3)$\implies$(1).
For this, assume that $\mal{B}$ satisfies the conditions in the statement and lies in $\mathscr{C}_{\geq m_0}$ for some $n_0$. Consider the fiber sequence in $\mathscr{C}$ 
$$\Sigma^{m_0}(\pi_{m_0}^{t}({Y}))\to \tau _{\leq m_0}^{}{Y}\to \tau^{}_{\leq m_0-1}{Y}\simeq 0.$$
The first term is given to be local, and the last term is tautologically local. We conclude that the extension $\tau_{\leq n_0}^{}{Y}$ is $L$-local. By repeating this argument, we see that all its truncations are local. Since, by left completeness, the morphism $${Y}\to \plim_{i}\tau^t_{\leq i}{Y}$$ is an equivalence and $L\mathscr{C}$ is closed under limits, we conclude that ${Y}$ is $L$-local. 

Finally, observe that if ${X}$ is such that all the objects $\pi_i^{s}{X}$ are local, then its connective covers $\tau_{\geq m}^{t}{X}$ also have the same property. But then the paragraph above says those covers must be local. To conclude, it suffices to note that, by right completeness, $$\dcolim_{m}\tau_{\geq m}^{t}{X}\to{X} $$ is an equivalence, and this will belong to $L\mathscr{C}$ by closedness under telescopic colimits.

(3)$\iff $(4) is formal form the identification $(L\mathscr{C})^\heartsuit = L\mathscr{C}\cap\mathscr{C}^\heartsuit$.
 \end{proof}

\begin{cor}\label[cor]{characterizing n bir spectra}
    Let $\esc{X}$ be a Nisnevich local spectrum over a field. The following are then equivalent:
    \begin{enumerate}
        \item $\esc{X}$ is $n$-tate truncated, i.e., $\esc{X}\simeq f_{0/n}\esc{X}$.
        \item For all $i$ the nisnevich connective covers $\tau^{nis}_{\geq i}\esc{X}$ are $n$-tate truncated.
        \item For all $i$ the stable homotopy group $\pi_i^{nis}\esc{X}$ are strictly $(n-1)$-spherical strictly $\mathbb{A}^1$-invariant.
    \end{enumerate}
    When $k$ is a perfect field, these are moreover equivalent to 
    \begin{enumerate}[resume]
        \item For all $i$ the stable homotopy groups $\pi_i^{nis}\esc{X}$ are strictly  $\mathbb{A}^1$-invariant and $(\pi_i^{nis}\esc{X})_{-n}=0$.
    \end{enumerate}
\end{cor}
\begin{proof}
    Apply the above Theorem in the presence of \Cref{slice connectivity}. Identify the heart as the strictly $n$-spherical, strictly $\mathbb{A}^1$-invariant sheaves of abelian groups from \Cref{t structure on n bir spectras}.
\end{proof}

\begin{prop}\label[prop]{non negative slice generation}
Over any field $k$, the $\infty$-category $(\mathcal{SH}^{S^1}(k)/{f_n})_{\geq 0}$ is the smallest subcategory generated by $\Sigma_+^{\infty}X/f_n$ under colimits as subcategories of:
\begin{enumerate}
    \item $\mathcal{SH}^{S^1}(k)/f_n$
    \item $\mathcal{SH}^{S^1}(k)$
    \item $\mathcal{SH}^{S^1}(k)_{\geq 0}$
\end{enumerate}
\end{prop}
\begin{proof}
      (1)  Let $\mathcal{C}\subset \mathcal{SH}^{S^1}(k)/f_n$ be the full subcategory generated by $\Sigma ^{\infty}_+X/f_n$. By the connectivity of the Tate truncation functor $f_{0/n}$ (\Cref{slice connectivity}), we have $\mathcal{C}\subset {(\mathcal{SH}^{S^1}(k)/f_n)}_{\geq 0}$. 
        Conversely, note that
        \begin{flalign*}
         (\mathcal{SH}^{S^1}(k)/{f_n})_{\geq 0}&=(\mathcal{SH}^{S^1}(k)/{f_n})\bigcap\mathcal{SH}^{S^1}(k)_{\geq 0}\\&=f_{0/n}(\mathcal{SH}^{nis}_{S^1}(k)_{\geq 0})\text{ [use the fact that }f_{0/n} \text{ is idempotent]}   
        \end{flalign*}
     But since $ \mathcal{SH}^{nis}_{S^1}(k)_{\geq 0}$ is generated under colimits by $\Sigma^\infty_+ X$ and $f_{0/n}$ preserves colimits, we see that $f_{0/n}(\mathcal{SH}^{nis}_{S^1}(k)_{\geq 0})\subset \mathcal{C}$. 

     (2) follows from (1) since the inclusion $\mathcal{SH}^{S^1}(k)/f_n\subset\mathcal{SH}^{S^1}(k)$ is closed under colimits (\Cref{slice localization is colimit closed}). (3) follows from the additional fact that $\mathcal{SH}^{S^1}(k)_{\geq 0}\subset \mathcal{SH}^{S^1}(k)$ is colimit closed.
\end{proof}

\subsection{Connectivity of Derived slices}
All constructions and computations adapt naturally to the setting of chain complexes. Generally, the principle is that we apply the whole construction to the linearization of  $Sm_S$, i.e., the additive category $\mathbb{Z}Sm_{S+}$ whose objects are $X$ for  $X\in Sm_S$ and the sets of morphisms are the free abelian group generated by the sets $Sm_{S}(X,Y)$: 
$$\mathrm{Hom}_{\mathbb{Z}{Sm}_{S}}(X, Y) = \mathbb{Z}[Sm_S(X, Y)].$$

Since $\mathbb{Z}Sm_{S}$ is an additive category, it follows that 
\begin{flalign}\label{pzsm=sabsm} 
\mathcal{P}_{\Sigma }(\mathbb{Z}Sm_{S})\simeq s\mathcal{A}b_\Sigma(Sm_S).
\end{flalign}
The the inclusion functor $$\mathbb{Z}:Sm_S\to \tilde{\mathbb{Z}}Sm_{S}$$ evidently factors through the inclusion $Sm_S\to 
Sm_{S+}$ as $$Sm_S \xrightarrow{(-)_+} Sm_{S+}\xrightarrow[]{\tilde{\mathbb{Z}}}\mathbb{Z}Sm_{S}$$ and hence induces the standard sequence of adjunctions by taking $\sqcup$-local derivation:$$\mathcal{P}_\Sigma(Sm_S) \rightleftarrows \mathcal{P}_\Sigma(Sm_S)_* \rightleftarrows \mathcal{P}_\Sigma(\mathbb{Z}Sm_S)\simeq  sA\mathrm{b}_\Sigma(Sm_S) \simeq \mathcal{D}_{\ge 0}(S)$$ where the equivalence at the middle is \Cref{pzsm=sabsm}, and the one on the extreme right is obtained by the Dold–Kan correspondence $s\mathcal{A}\mathrm{b} \simeq \mathcal{D}_{\ge 0}(\mathbb{Z})$.

 Passing to Nisnevich-local and $\mathbb{A}^1$-invariant sheaves gives equivalences: $$\mathcal{P}_{nis}(\mathbb{Z}Sm_S)\simeq \mathcal{P}_{nis}(Sm_{S}, Ch_{\geq 0}),\quad \mathcal{H}^{\mathbb{A}^1}(\mathbb{Z}Sm_S)\simeq {\mathcal{D}}^{\mathbb{A}^1}_{\geq 0}.$$ 

Upon stabilization, these equivalences give rise to equivalences:
$$\mathcal{S}pt_{nis}(\mathbb{Z}Sm_S)\simeq \mathcal{P}_{nis}(Sm_{S}, Ch):=\mathcal{D}_{nis}(S),\quad \mathcal{SH}^{S^1}(\mathbb{Z}Sm_S)\simeq {\mathcal{D}}^{\mathbb{A}^1}(S),$$
where $\mathcal{D}^{\mathbb{A}^1}(S)$ was defined in [\cite{morel2012a1}, Definition 6.17].

The connective adjunction then extends to the stable Hurewicz adjunction: $$\mathcal{SH}^{S^1}(S){\leftrightarrows}  \mathcal{SH}^{S^1}(\mathbb{Z}Sm_S)\simeq{\mathcal{D}}^{\mathbb{A}^1}(S):\mathrm{H}.$$  We denote the corresponding localization as $$L_{nis}^{ch}: {\mathcal{D}}(S)\to \mathcal{D}^{nis}(S),\quad L_{mot}^{ch}: {\mathcal{D}}(S)\to \mathcal{D}^{\mathbb{A}^1}(S).$$ This yields the following canonical commutative squares:
\[
\xymatrix@C=4pc@R=4pc{
  \mathcal{S}pt(S) \ar@<-0.7ex>[d]_{\Gamma} \ar@<0.7ex>[r]^{L_{{nis}}^{sp}} 
  & \mathcal{S}pt_{nis}(S) \ar@<-0.7ex>[d]_{\Gamma}\ar@<0.7ex>[r]^{L_{mot}^{sp}} \ar@<0.7ex>[l]  
  & \mathcal{SH}^{S^1}(S) \ar@<-0.7ex>[d]_{\Gamma} \ar@<0.7ex>[l]\\
  \mathcal{D}(S) \ar@<-0.7ex>[u]_{\mathrm{H}} \ar@<0.7ex>[r]^{L_{nis}^{{ch}}} 
  & \mathcal{D}_{nis}(S) \ar@<-0.7ex>[u]_{\mathrm{H}}\ar@<0.7ex>[r]^{L_{mot}^{\mathrm{ch}}}\ar@<0.7ex>[l] 
  & \mathcal{D}^{\mathbb{A}^1}(S) \ar@<-0.7ex>[u]_{\mathrm{H}} \ar@<0.7ex>[l]
}\]

\begin{ex}
If $k$ is a field of characteristic $0$, then taking global sections of the de Rham complex $$\Omega_{dR}:=\Gamma(-,\Omega_{-/k}^\bullet) : Sm_k^{op}\to Ch_{\geq 0}$$ is a motivic complex. Indeed, for $\mathbb{A}^1$-invariance, recall that when $X/k$ is smooth and separated, there is a Kunneth isomorphism  $$\Gamma (\mathbb{A}^1_k\times _k X,\Omega_{\mathbb{A}^1_k\times _k X}^\bullet)\simeq \Gamma (\mathbb{A}^1_k,\Omega_{\mathbb{A}^1_k}^\bullet )\otimes_k \Gamma (X,\Omega_{X/k}^\bullet)$$ [\cite[\href{https://stacks.math.columbia.edu/tag/0FMB}{Lemma 0FMB}]{stacks-project}], while over characteristic zero $$\Gamma (\mathbb{A}^1_k,\Omega_{\mathbb{A}^1_k}^\bullet )\simeq k$$ by the algebraic Poincar\'e lemma [\cite{hartshorne1975rham}, Proposition (7.1)]. Together, it follows that $$\Gamma (\mathbb{A}^1_k\times _k X,\Omega_{\mathbb{A}^1_k\times _k X}^\bullet)\simeq\Gamma (X,\Omega_{X/k}^\bullet).$$ Since the base is affine, for Nisnevich descent it suffices to consider affine Nisnevich squares. This follows from [\cite{cisinski2012mixed}, 3.1.3].
\end{ex}
\begin{ex}
Every strictly $\mathbb{A}^1$-invariant Nisnevich sheaf of abelian groups, thought of as a complex of degree $0$, is a motivic complex.
\end{ex}

To introduce the derived slice filtration, we set $$\mathbb{Z}(n):=cofib(\oplus \mathbb{Z}\mathbb{G}_m^{n-1}\to \mathbb{Z}\mathbb{G}_m^n)[-n].$$ We denote the full subcategory of $\mathcal{D}^{\mathbb{A}^1}(S)$ monoidally generated by $\mathbb{Z}(n)$ as $\mathcal{D}^{\mathbb{A}^1}(S)(n)$. This is a colocalizing subcategory, and we denote its colocalization by $f^{ab}_n$. The corresponding localization is denoted $f^{ch}_{0/n}$, and the localizing subcategory is denoted $\mathcal{D}^{\mathbb{A}^1}(S)/f^{ch}_{n}$. As in the spectral setup, there is a tower of exact localizations of the stable $\infty$-category of effective motivic complexes:
$$\cdots\to f^{ch}_{0/n+1}\to f^{ch}_{0/n}\to \cdots \to f^{ch}_{0}=1$$
The adjunctions above then descends to $$\Gamma_{/n}: \mathcal{SH}^{S^1}(S)/f^{}_{n}{\leftrightarrows} \mathcal{D}^{\mathbb{A}^1}(S)/f^{ch}_{n}: N_{/n}$$ such that $$\Gamma_{/n}f_{0/n}\simeq f^{ch}_{0/n}\Gamma.$$
\begin{rem}
    Since the category $\mathcal{D}(Sm_k)$ is stable, it is not hard to see that all of Levine's ideas from the homotopy coniveau tower paper [\cite{levine2008homotopy}] work very well for $\mathcal{D}(Sm_k)$, and we have a model for $f_{0/n}^{ch}$ when $k$ is an infinite perfect field. Namely, use $f^{ch}_{0/n}:=(-)^{0/n}\circ L_{mot}^{ch}$ where $$\esc{E}^{(0/n)}(X)=\mathrm{Tot}^{\oplus}\big (\colim_{Z\in S^{(0)}(X,\bullet), W\in S^{(n)}(X,\bullet)}\esc{E}^{Z\setminus W}(X\times \Delta ^\bullet\setminus W) \Big )$$ and $$L^{ch}_{mot}(\esc{E}):= \mathrm{Tot}^{\oplus}(\esc{E}^{\mathbb{Z}({\mathbb{A}^{\bullet}/0} )}).$$
\end{rem}
\textbf{Slice connectivity of chain complexes}
One can carry over all the results from the $\mathcal{SH}^{S^1}(k)$, as performed in the previous sections, to the derived settings. We record them without proof: 

\begin{thm}
    Over a field $k$, the linear Tate truncation $f_{0/n}^{ch}$ functor preserves homological connectivity for every $n\geq 0$.
\end{thm}

\begin{thm}\label{t structure on n spherical motivic complexes}
    Let $k$ be a field and $n\geq 0$. The $n$-th Tate-truncated category $\mathcal{D}^{\mathbb{A}^1}/f_{n}$ has an accessible, complete t-structure. The inclusion functors $$\mathcal{D}^{\mathbb{A}^1}(k)/f_n^{ch}\subset \mathcal{D}^{\mathbb{A}^1}(k),\quad \text{ and }\mathcal{D}^{\mathbb{A}^1}(k)/f_n^{ch}\subset \mathcal{D}^{nis} (k)$$ are $t$-exact. The heart of this $t$-structure is equivalent to ${Ab^{\mathbb{A}^1}_k}{/f_n}$. When $k$ is a perfect field, this is equivalent to the kernel of the functor $$(-)_{-{n}}: Ab_k^{\mathbb{A}^1}\to Ab_k^{\mathbb{A}^1}.$$ 
\end{thm}
As a consequence of all this, we obtain the slice Hurewicz theorem. To state it, we introduce a few notations. 
\begin{defn}
For a $k$-space $\esc{X}$, we define its $S^{p,n}$-null homotopy groups as:
$$\pi_i^{p,n}\esc{X}:=\pi_i^{nis}L^{p,n}\esc{X},$$ where $L^{p,n}$ is the $S^{p,n}$-nullification functor of \textup{[\cite{asok2023p}, \S3]}. We define its $n$-slice  homology groups as $$H^{0/n}_i\esc{X}:=H_i^{nis}f_{0/n}^{ch}\Gamma \esc{X}.$$
\end{defn}
The following is straightforward:
\begin{lem}
    $B\mapsto H^{0/n}_0B$ is the left adjoint to the inclusion ${Ab^{\mathbb{A}^1}_k}{/f_n}\subset Sh_k^{\mathbb{A}^1}$ which we denote by $\mathbb{Z}^{0/n}$. Similarly, $A\mapsto H^{nis}_0f^{ch}_nA$ is the localization ${Ab^{\mathbb{A}^1}_k}{/f_n}\subset Ab_k^{\mathbb{A}^1}$ which we denote by $(-)/f_n$.
\end{lem}
\begin{ex}\label[ex]{gw/f1=z}
   We claim that $\underline{\mathrm{GW}}/f_1=\mathbb{Z}$. In other words, the rank map $rk:\underline{\mathrm{GW}}\to\mathbb{Z}$ is the universal morphism to a sheaf of abelian groups with trivial $\mathbb{G}_m$-contraction. It is clear that $\mathbb{Z}$ has a trivial $\mathbb{G}_m$-contraction, so we only need to address its universal property. Let $G$ be a strongly $\mathbb{A}^1$-invariant sheaf of abelian groups with $G_{-1}=0$, and let $f:\underline{\mathrm{GW}}\to G$ be a group homomorphism. We want to show that $f$ factors uniquely through the rank map. Since the rank map is the cokernel of the inclusion $I\subset \underline{\mathrm{GW}}$, it suffices to show that $f(I)=0$. The subgroup $I$ is precisely the image of the Hopf map $\eta:K^{MW}_1\to \underline{\mathrm{GW}}$ (indeed, over fields, $I$ is generated by Pfister elements $\langle\langle a\rangle\rangle:=\langle a\rangle-1$, which is identical to $\eta[a]$ under the notation of [\cite{morel2012a1}, Chapter 3]). Thus, it suffices to show that the composite group homomorphism
$$K^{MW}_1\xrightarrow{\eta} \underline{\mathrm{GW}}\xrightarrow{f} G$$
is identical to the zero map. Since $G\in Ab^{\mathbb{A}^1}_k$ by [\cite{morel2012a1}, Theorem 3.37], the above morphism is uniquely determined by a morphism of sheaves $\mathbb{G}_m^{\wedge 1}\to G$. Since the sheaves are discrete, this last morphism is uniquely determined by an element of $G_{-1}$, which is trivial by assumption.
\end{ex}
  \begin{thm}[motivic $n$ spherical hurewicz theorem]\label{slice hurewicz theorem}
Let $k$ be a perfect field and $\esc{X}$ be a pointed $n$-birationally $(m-1)$-connected space ($m\geq 0$). Then the Hurewicz map $$\pi^{p,n}_m\esc{X}\to H^{0/n}_m\esc{X}$$ is an isomorphism for $m\geq \max{(p-n,2)}$; for $1\leq m<p-n$ it is the universal strictly $\mathbb{A}^1$-invariant, strictly $n$-spherical (i.e., having trivial $n$-fold $\mathbb{G}_m$ contraction) sheaf of abelian groups associated to the sheaf of groups $\pi^{p,n}_m\esc{X}$, and for $m=0$ it is the free such sheaf of abelian groups generated by $\pi_0^{p,n}$. That is, $$H^{0/n}_0\esc{X}\cong H^{0/n}_0\pi_0^{p,n}\esc{X}.$$
\end{thm}
\begin{proof}
We note that the composite functor $C^{0/n}_*:=f_{0/n}^{ch}C_*=f^{ch}_{0/n}\Gamma\Sigma^{\infty}$ preserves connectivity because all its component functors do. Also, a sheaf $A$ of abelian groups is strictly $n$-spherical iff $\mathrm{K}^{nis}(A,m)$ is an $L^{p,n}$-local motivic space for all $m$ (see \Cref{n spherical is lpn null}). 
    
We have the following chain of equivalences for any strictly $n$-birational sheaf of abelian groups $A$, 
\begin{flalign*}
   m \mbox{-}Grp^{nis}_k( \pi_m^{p,n}\esc{X},A)&\cong m\mbox{-}Grp^{nis}_k(\pi_m^{nis}L^{p,n}\esc{X},A)\\&\cong\mathcal{P}_\bullet(L^{p,n}\esc{X}, \mathrm{K}^{nis}(A,m)) \text{ [by the adjunction } \pi^{nis}\dashv \mathrm{K}^{nis}]\\&\cong\mathcal{H}_\bullet^{p,n}(L^{p,n}\esc{X}, \mathrm{K}^{nis}(A,m)) \text{ [since } \mathrm{K}^{nis}(A,m) \text{ is } L^{p,n} \text{ local}]\\&\cong \mathcal{D}^{\mathbb{A}^1}(k)/{{f^{ch}_n}}(C^{0/n*}\esc{X},A[m]) \text{ [by adjunction]}\\&\simeq {Ab^{\mathbb{A}^1}_k}{/f_n}(H_m^{nis}(C^{0/n*}\esc{X}),A)\text{ [since } C^{0/n*}\esc{X}\in \mathcal{D}^{\mathbb{A}^1}(k)/{{f^{ch}_n}}_{\geq 0}]\\&= {Ab^{\mathbb{A}^1}_k}{/f_n}(H^{0/n}_m\esc{X},A).
\end{flalign*}
The result now follows from the Yoneda lemma for $m\geq 2$ and from the appropriate free-forgetful adjunctions from sets to groups for $m=0$ and from the abelianization adjunction from groups to abelian groups for $m=1$. For $m\geq p-n,2$, the point is that $\pi_m^{p,n}\esc{X}$ has trivial $n$-fold $\mathbb{G}_m$-contraction [\cite{asok2023p}, Corollary 3.1.21].
\end{proof}
\section{Consequences of slice connectivity}

\subsection{Stable $n$-Birational connectivity}\mbox{}

\underline{\textbf{Stable \texorpdfstring{$n$}{n} birational motives \texorpdfstring{$\mathcal{SH}^{n}$}{sh}:}}
The $n$-birational motivic homotopy category of $S$ is defined as the localization of the motivic homotopy category of $S$ with respect to the set $B(n)$ of $n$-dense open immersions in $Sm_S$. This category was first introduced in this generality in [\cite{bachmann2019voevodsky}]. (In [\cite{sfbat}], the author has studied its functorial properties in detail.) In this section, we shall study its stable counterpart and prove a stable connectivity theorem. Unfortunately, we shall do this only over perfect fields, identifying the stable $n$-birational localization with $f_{0/n+1}$, since it is not clear at the moment whether this identification holds over arbitrary fields.

Let us first gather the required machinery. As usual, the stable $n$-birational motivic homotopy category is defined as the stabilization of $\mathcal{H}^{n}(S)$ i.e. it is the infinite colimit 
$$\mathcal{H}^{n}(S)_\bullet\xrightarrow[]{\Sigma}\mathcal{H}^{n}(S)_\bullet\xrightarrow[]{\Sigma}\mathcal{H}^{n}(S)_\bullet\xrightarrow[]{\Sigma}\cdots \mathrm{Stab}(\mathcal{H}^{n}(S))$$ in $Pr^L$, which is thus the limit:
$$\mathrm{Stab}(\mathcal{H}^{n}(S))\dots\xrightarrow[]{\Omega}\mathcal{H}^{n}(S)_\bullet\xrightarrow[]{\Omega}\mathcal{H}^{n}(S)_\bullet\xrightarrow[]{\Omega}\mathcal{H}^{n}(S)_\bullet$$
in $Cat_\infty$ . Since fully faithful functors are stable under limits, the induced morphisms 
\begin{flalign}
\mathcal{SH}^{n}(S))&\xhookrightarrow{}\mathcal{SH}^{S^1}(S)\xhookrightarrow{}\mathcal{S}pt_{nis}(S)
\end{flalign}
 are fully faithful. We thus see that $\mathrm{Stab}(\mathcal{H}^{n}(S))$ is a localization of $\mathcal{SH}^{S^1}(S)$, and we denote this localization by $L_{bir}^{n,sp}$. The saturated class of morphisms in this localization is generated by $\Sigma^{\infty }_+U\to \Sigma^{\infty }_+X$ where $U\to X$ is a dense open immersion in $Sm_S$. 
 
We can, of course, do better: $$\mathrm{Stab}(\mathcal{H}^{n}(S))\simeq\mathcal{H}^{n}_\bullet(S)[(S^1)^{-1}].$$ In fact, since base change along smooth morphisms preserves $n$-dense open immersions, $L_{n}^{st}$ is a monoidal localization functor (stabilize [\cite{sfbat}, Theorem 2.5.5.]).

\begin{rem}\label[rem]{An-0 loc is slice}
Since $\mathbb{A}^{n+1}\setminus 0\hookrightarrow \mathbb{A}^{n+1}$ is an $n$ dense open immersion, it is evident that $\mathcal{SH}^{n}_{S^1}(S)\subset L^{st}_{\mathrm{S}^{n}_\mathbb{A}}\mathcal{SH}^{\mathbb{{A}^1 }}_{S^1}(S)$. By \Cref{tate truncation is spherical localization} it thus follows that the birational localization functors $L_{bir}^{n,sp}: \mathcal{SH}^{S^1}(S)\to \mathcal{SH}^{n}_{S^1}(S)$ factor through $$f_{0/n+1}:\mathcal{SH}^{S^1}(S)\to L^{st}_{\mathbb{A}^{n+1}\setminus0}\mathcal{SH}^{\mathbb{{A}^1 }}_{S^1}(S)\simeq\mathcal{SH}^{S^1}(S)/f_{n+1}.$$ 
\end{rem}
\begin{lem}[Generic Smoothness]\label[lem]{GENERIC SMOOTHNESS}
    Let $k$ be a perfect field. The class of stable $n$-birational equivalences is the saturation class generated by $n$-dense open immersions whose complements are smooth and have trivial normal bundles.
\end{lem}

\begin{proof}
Given a dense open immersion $U\hookrightarrow X$ of smooth schemes over $k$ (if the complement $X\setminus U$ is not smooth), one first finds a filtration of $X$ by open subschemes $U=U_n\subset U_{n-1}\subset \cdots\subset U_1\subset U_0=X$ such that, for all $i$, $U_{i} \setminus U_{i-1} =C_i$ is smooth and has trivial normal bundle in $U _i$ [\cite{levine2010slices}, Remark 5.2]. If $cod_{X}(X\setminus U)\ge d$, it is clear that $cod_{U_i}(U_i\setminus C_i)\geq d$. Since saturation classes are closed under composition, the morphism $U\hookrightarrow X$ belongs to the given saturation class.
\end{proof}

\begin{thm}\label{bir spectras are 0 slices}
  Let $k$ be a perfect field. The induced map $$L_{bir}^{n,sp}:  \mathcal{SH}^{S^1}(k)/f_{n+1}\to \mathcal{SH}^n(k) $$ is an equivalence of stable infinity categories leading to the identification $f_{0/n+1}\simeq L_{bir}^{n,sp}$.
\end{thm}
\begin{proof}
The localization functor $$f_{0/n+1}:\mathcal{SH}^{S^1}(k)\to \mathcal{SH}^{S^1}(k)/f_{n+1}$$ inverts birational local equivalences . Indeed, since $f_{0/n+1}$ is an exact localization by construction, it suffices to show that $f_{0/n+1}$ takes co-fibers of $n$-dense open immersions to equivalences. But over perfect fields, using \Cref{GENERIC SMOOTHNESS}, it suffices to consider $n$-dense open sub-schemes with regular complements with trivialization of normal bundles. In such cases, by the motivic purity theorem, we are reduced to showing that $$s_nTh_Z(N_ZX)\simeq s_nT^{\wedge n}$$ is contractible. This is, of course, true by definition.

  Consequently, $f_{0/n+1}$ factors through $\mathcal{SH}^n(k)$. This provides the inverse of the morphism in question. 
\end{proof}

\begin{rem}\label[rem]{n as S^n and A^1}
    A quick presentation of the above theorem is that the following squares are cartesian in $\text{Cat}_\infty^{ex}$ for a perfect field $k$:
    \[
    \xymatrix{
    \mathcal{SH}^{n}_{S^1}(k)\ar@{}[r]|{\rotatebox{-90}{$\bigcup$}}\ar@{}[d]|{\rotatebox{180}{$\bigcup$}}&L^{st}_{(\mathrm{S}^{n}_\mathbb{A})}\mathcal{SH}_{S^1}(k) \ar@{}[d]|{\rotatebox{180}{$\bigcup$}}\\
    \mathcal{SH}^{S^1}(k) \ar@{}[r]|{\rotatebox{-90}{$\bigcup$}}&L_{nis}^{ab}\mathcal{SH}_{S^1}(k) }
    \]

\end{rem}
\begin{cor}\label[cor]{over perfect bn localization is f0/n+1}
    Over a perfect field, we have $L_{bir}^{n,sp}\simeq f_{0/n+1}L_{mot}^{sp}\simeq L^{sp}_{\mathbb{A}^1,\mathbb{A}^{n+1}\setminus0, nis}\simeq \dcolim_{k}\big(\Phi_{\mathrm{S}^n_\mathbb{A}}^\infty \Phi_{\mathbb{A}^1}^\infty L_{nis}^{sp}\big )^k\simeq \dcolim_{k}\big(\Phi_{\mathrm{S}^n_\mathbb{A}} \Phi_{\mathbb{A}^1}\big )^k L_{nis}^{sp}$. 
\end{cor}
\begin{proof}
    This follows at once from \Cref{bir spectras are 0 slices} and \Cref{formula for a1 sn nis localization}.
\end{proof}

\begin{cor}\label[cor]{S1 biraqtional model}
  Let $k$ be an infinite perfect field. We have the following exact birational models for presheaves of spectra on $Sm_k$:
  $$L_{bir}^{n,sp}\simeq (L_{mot}^{sp})^{^{(0/n+1)}}\simeq \dcolim_{m}\big(\Phi_{\mathrm{S}^{n}_\mathbb{A}} \Phi_{\mathbb{A}^1} \big )^mL_{nis}^{sp}.$$
\end{cor}
\begin{proof}
   Using the above corollary, it is a restatement of [\cite{levine2008homotopy}, Theorem 7.1.1]. 
\end{proof}
\begin{cor} \label{n birational stable connectivity}
    Over a perfect field, $L_{bir}^{n,sp}$ preserves connectivity.
\end{cor}
\begin{proof}
    Using \Cref{bir spectras are 0 slices}, this follows from the slice connectivity theorem: \Cref{slice connectivity}.
\end{proof}

\begin{rem}\label[rem]{rem about n=0}
    While our general strategy utilizes the stable algebraic topology of the slice filtration to deduce the same for the $n$-birational localization over perfect fields, the $n=0$ case (localizing at all dense open immersions) admits an independent treatment. Notably, for $n=0$, the motivic framework extends beyond perfect fields, yielding a $0$-birational stable connectivity theorem valid over arbitrary qcqs base schemes [\cite{bas}, Theorem 4.1.7].
\end{rem}

\begin{cor}\label[cor]{t structure on n bir spectras}
    Let $k$ be a perfect field. Then the stable $\infty$-category of $n$-birational motivic spectra has a $t$-structure whose nonnegative part is generated under colimit by $n$-birational models of smooth $X/k$ and such that the inclusion $\mathcal{SH}^n_{S^1}(k)\subset \mathcal{SH}_{S^1}(k)$ is right $t$-exact. The heart of this $t$-structure is equivalent to the category of strictly $n$-birational sheaves of abelian groups. 
\end{cor}
Here, a strictly $n$-birational sheaf of abelian groups is a Nisnevich sheaf all of whose Nisnevich cohomology groups are $\mathbb{A}^1$-local and $n$-dense local. It follows from \Cref{bir spectras are 0 slices} and \Cref{t structure on n spherical motivic spectras} that, over perfect fields, a Nisnevich sheaf is strictly $n$-birational if and only if its $(n+1)$-fold $\mathbb{G}_m$-contraction is trivial.
\begin{cor}
Let $k$ be a perfect field. A Nisnevich sheaf of spectra $\mal{A}$ is $n$-birational iff it is the $(n+1)$-th layer of Voevodsky's slice tower iff all of its Nisnevich connective covers $\tau_{\geq n}^{nis}\esc{X}$ are $n$-birational local iff all of its Nisnevich stable homotopy groups are strictly $n$-birational and strictly $\mathbb{A}^1$-invariant.
\end{cor}

\textbf{\texorpdfstring{$n$}{}-Birational motivic complexes.}
Similarly, we define $${\mathcal{D}}^{n}(S):=L_{bir}^{n,ch}{\mathcal{D}}^{\mathbb{A}^1}(S)$$ as the localization at the set $\mathbb{Z}B(n)$, thought of as morphisms of complexes of degree 0. This clearly comes with a canonical adjunction $$\Gamma_n : \mathcal{SH}^{n}(S){\leftrightarrows} \mathcal{D}^{n}(S): \mathrm{H}$$ 

One can carry over all the results from the ordinary stabilization, as performed in the previous sections, to the derived settings. We record them without proof: 

\begin{thm}\label{bir complexes are 0 slices}
  Let $k$ be a perfect field. The induced functor $$  \mathcal{D}^{\mathbb{A}^1}(k)/f_{n+1}\to \mathcal{D}^{n}(k) $$ is an equivalence of stable infinity categories. Thus, we obtain the canonical equivalence $f^{ch}_{0/n+1}\simeq L^{ch}_{n}$.
\end{thm}
In other words, as in \Cref{n as S^n and A^1}, we have a pullback square of stable infinity categories:  
\[
    \xymatrix{
    \mathcal{D}^{n}_{eff}(k)\ar@{}[r]|{\rotatebox{-90}{$\bigcup$}}\ar@{}[d]|{\rotatebox{180}{$\bigcup$}}&L^{ab}_{(\mathrm{S}^{n}_\mathbb{A})}\mathcal{D}(k) \ar@{}[d]|{\rotatebox{180}{$\bigcup$}}\\
 \mathcal{D}^{\mathbb{A}^1}_{eff}(k) \ar@{}[r]|{\rotatebox{-90}{$\bigcup$}}&L_{nis}^{ab}\mathcal{D}(k) 
    }
    \]
    The $0$-birational case of this square, though with transfers (and with $\mathbb{P}^1$ instead of $\mathbb{G}_m$) was obtained by J. Ayoub in [\cite{ayoub2020P1}, Proposition 3.9].
\begin{thm}
   Let $k$  be a perfect field. Then the equivalence $$\mathcal{D}^{\mathbb{A}^1}(k)/f_{n+1}  \overset{L_{bir}^{n,ch}}{\underset{\simeq}{\longrightarrow}}\mathcal{D}^{n}(k)$$ induces an accessible complete $t$ structure with heart, the category of strictly $n$ birational strictly $\mathbb{A}^1$ invariant nisnevich sheaves of abelian groups.
\end{thm}

For the next theorem, we define the notion of $n$-birational motivic homotopy groups of spaces. 
\begin{defn}
    Given a $k$-space $\esc{X}$ we define its $n$-birational motivic homotopy group as:
    $$\pi_i^{n\mbox{-}bir}\esc{X}:=\pi_i^{nis}L^n_{bir}\esc{X}.$$ We define the $n$-birational homology group of $\esc{X}$ as $$H^{n\mbox{-}bir}_i\esc{X}:=H^{nis}_i(L_{bir}^{n,ch}\Gamma\esc{X}).$$
\end{defn}
\begin{rem}
    When $k$ is a perfect field, it follows from  that $H^{n\mbox{-}bir}_i\esc{X}\simeq H^{p,n+1}_i\esc{X}$. However, such an identification fails for $\pi_i$'s, i.e., in general $\pi_i^n\not\simeq\pi_i^{p,n+1}$. For example, when $n=0$ and $p=2$ we get $\pi_0^{2,1}(\mathbb{G}_m)=\mathbb{G}_m$ since $\mathbb{G}_m $ is both $\mathbb{A}^1$-local as well as $\mathbb{P}^1$-local. On the other hand $\pi_0^{0\mbox{-}bir}(\mathbb{G}_m)\simeq *$ since $\mathbb{G}_m$ is birational to the (by definition) $L^0$-contractible scheme $\mathbb{A}^1$. In an earlier paper, the author shows that when $k$ is perfect, the identity $\pi_i^{0\mbox{-}bir}\simeq\pi_i^{2,1}\simeq \pi_i^{1,1}$ does work for motivically connected presheaves of spaces [\cite{0bat}, Corollary 3.4.6.].
\end{rem}
 \begin{thm}[$n$-Birational hurewicz theorem]\label{n Birational hurewicz theorem}
Let $k$ be a perfect field and $\esc{X}$ be a pointed $n$-birationally $(m-1)$ connected space ($m\geq 0$). Then the hurewicz map $$\pi^{n\mbox{-}bir}_m\esc{X}\to H^{n\mbox{-}bir}_m\esc{X}$$  is an isomorphism for $m\geq \max{(n,2)}$, is the universal strictly $\mathbb{A}^1$ invariant strictly $n$-birational sheaf of abelian groups associated to the sheaf of groups $\pi^{n\mbox{-}bir}_m\esc{X}$ for $1\leq m<n$, and for $n=0$ is the free such sheaf of abelian groups generated by $\pi_i^{n}$, i.e. $$H^{n\mbox{-}bir}_i\esc{X}\cong H^{n\mbox{-}bir}_0\pi_i^{n}\esc{X}.$$
\end{thm}

\subsection{Slice connectivity of \texorpdfstring{$\mathbb{P}^1$}{}-spectras}
In this subsection, we aim to show that the slice filtration on $\mathbb{P}^1$-spectra also satisfies a suitable connectivity property. Our main insight relies on the slice connectivity theorem for $S^1$-spectra and the confirmation of the slice conjectures. Before presenting the main result, let's clarify the key terms mentioned so far in this paragraph, beginning with the term $\mathbb{P}^1$-spectra.

The motivic stable $\infty$-category of $\mathbb{P}^1$-spectra is defined as the monoidal inversion of the object $\mathbb{P}:=(\mathbb{P}^1,1)$, i.e.,
$$\mathcal{SH}^{\mathbb{P}^1}:=\mathcal{H}^{\mathbb{A}^1}[\mathbb{P}^{-1}].$$ Since $\mathbb{P}\simeq \Sigma\mathbb{G},$ it follows that 
$$\mathcal{SH}^{\mathbb{P}^1}\simeq \mathcal{SH}^{S^1}[\mathbb{G}^{-1}]\simeq \mathcal{SH}^{S^1}[\mathbb{P}^{-1}].$$ Canonically, these come equipped with the canonical inversion functors:
$$\Sigma^\infty_{\mathbb{P}}:\mathcal{H}^{\mathbb{A}^1}\to \mathcal{SH}^{\mathbb{P}^1}.$$

The effective motivic stable $\infty$-category, $\mathcal{SH}^{\text{eff}}$, is the localizing stable subcategory of $\mathcal{SH}^{\mathbb{P}^1}$ generated by $\Sigma^\infty_{\mathbb{P}}(\mathcal{H}^{\mathbb{A}^1})$. Smashing with powers of $T$ yields a $\mathbb{Z}$-indexed filtration of compactly generated, colimit-closed monoidal stable sub-$\infty$-categories:$$\cdots \subset \mathcal{SH}^{\text{eff}}\wedge T^{n} \subset \cdots \subset \mathcal{SH}^{\text{eff}} \subset \mathcal{SH}^{\text{eff}}\wedge T^{-1} \subset \cdots$$By presentability, the inclusions admit right adjoints $f_n$. This yields the effective tower, from which Voevodsky defines the $n$-th slice functor $s_n$ as the cofiber of $f_{n+1}\to f_n$.

Clearly, when $n\geq 0$ the colocalizing functors $f_n:\mathcal{SH}^{\mathbb{P}^1}\to\mathcal{SH}^{\mathbb{P}^1}(n) $ factors through $\mathcal{SH}^{eff}$. We now want to restate the connectivity result of \S2 for $\mathbb{P}^1$-spectras. But as we have seen in \S4 this can equivalently be stated in the language of $t$-structures. We shall therefore recall the canonical $t$-structure on $\mathcal{SH}^{eff}(k)$ next. 

\begin{prop}\label[prop]{eff t structure}
    There is a $t$-structure on $\mathcal{SH}^{eff}(k)$ whose positive part is given by \begin{flalign*}
        \mathcal{SH}^{eff}(k)_{\geq 0}:=\{\mal{E}\mid \pi_i^{nis}\mal{E}=0 \text{ for all } i<0\}\\
         \mathcal{SH}^{eff}(k)_{\leq 0}:=\{\mal{E}\mid \pi_i^{nis}\mal{E}=0 \text{ for all } i>0\}.
    \end{flalign*}
\end{prop}
\begin{prop}\label[prop]{effective slice exact}
Let $k$ be a perfect field and $n\geq 0$ a natural number. Then

\begin{enumerate}
    \item The localization functor $f_{0/n}:\mathcal{SH}^{eff}(k)\to \mathcal{SH}^{eff}(k)$ is right $t$-exact inducing a $t$-structure on $\mathcal{SH}^{eff}(k)/f_n$.
    \item The colocalization functor $f_{n}:\mathcal{SH}^{eff}(k)\to \mathcal{SH}^{eff}(k)$ is right $t$-exact of amplitude $-1$.
     \item The functor $s_{n}:\mathcal{SH}^{eff}(k)\to \mathcal{SH}^{eff}(k)$ is right $t$-exact of amplitude $-1$.
\end{enumerate}
\end{prop}
\begin{proof}
    First, note that by definition, $\mal{E}\in   \mathcal{SH}^{eff}(k)_{\geq 0}$ if and only if $\omega_{\mathbb{G}}^{{\infty}}\mal{E}\in \mathcal{SH}_{S^1}(k)_{\geq 0}$. Now suppose $\mal{E}\in   \mathcal{SH}^{eff}(k)_{\geq 0}$. Then, by the validity of the slice conjecture over a perfect base field [\cite{bachmann2019voevodsky}, Corollary 19 (i)], we have \begin{flalign*}
        \omega_{\mathbb{G}}^\infty (f^{eff}_{0/n}\mal{E})\simeq f_{0/n}\omega_\mathbb{G}^\infty\mal{E}
    \end{flalign*}
   Since $\mal{E}\in   \mathcal{SH}^{eff}(k)_{\geq 0}$ and $f_{0/n}:\mathcal{SH}^{S^1}(k)\to \mathcal{SH}^{S^1}(k)$ is right $t$-exact (\Cref{slice t-exact}), $f_{0/n}\omega_\mathbb{G}^\infty\mal{E}\in \mathcal{SH}_{S^1}(k)_{\geq 0}$. Thus $ \omega_{\mathbb{G}}^\infty (f^{eff}_{0/n}\mal{E})\in \mathcal{SH}_{S^1}(k)_{\geq 0}$ and the claim follows from the first line of this proof again.
\end{proof}
\begin{rem}
    It is easy to show that the non-negative part of this $t$-structure on $\mathcal{SH}^{eff}(k)/f_n$ form \Cref{effective slice exact}(1) is the smallest subcategory generated under colimits by $f_{0/n}\Sigma^\infty_{eff}X_+.$
\end{rem}
As a corollary, we obtain:
\begin{thm}\label{slice conn for p1 spectra}
  For every $n\geq 0$, the localization $f_{0/n}: \mathcal{SH}^{\mathbb{P}^1}(k)\to \mathcal{SH}^{\mathbb{P}^1}(k)$ is right $t$-exact.
\end{thm}
\begin{proof}
    Recall that the colocalization $f_0:\mathcal{SH}^{\mathbb{P}^1}(k)\to \mathcal{SH}^{eff}(k)$ is $t$-exact for the $t$-structure on $\mathcal{SH}^{\mathbb{P}^1}(k)$ induced under $\Omega^\infty_\mathbb{G}$ and the $t$-structure on $\mathcal{SH}^{eff}(k)$ given by \Cref{eff t structure} [\cite{MR3743071}, Proposition 4(3)]. The claim then follows from \Cref{effective slice exact}(1).
\end{proof}

\section{Slice filtration on mixed motives with correspondence}
In this section, we discuss the connectivity property of the slice filtration for motives with correspondences. In contrast to the difficulty of establishing the connectivity theorem for both spectra and chain complexes, we shall show that it is a fairly easy task when the correspondence has a cancellation property. Moreover, in the case of spectra and complexes over the ordinary category $Sm_k$, the connectivity properties we have established are restrictive in the sense that while the connectivity amplitude of $f_{0/n}$ is $0$, the same for $f_n$ and $s_n$ is generally $-1$ (with the exception of $n=0$). We shall see in this section that for categories of mixed motives with cancellation, all of these functors have connectivity amplitude $0$.

We first recall the abstract machinery for the slice filtration on a category of correspondences. Let $\gamma: Cor_k^{fr}\to \mathcal{C} $ be a category of correspondences over $k$ [\cite{bachmann2021cancellation}, Definition 2.1]. The stable motivic homotopy category of $S^1$-spectra is then constructed as $$\mathcal{SH}^{\mathcal{C},S^1}(k):=L_{\mathbb{A}^1}\mathcal{P}_{nis}(\mathcal{C}, Spt).$$ We let $\mathbb{G}_\mathcal{C}:=\gamma(\mathbb{G}_m,1)$.
The $n$-Tate-twisted category of motivic spectra over $\mathcal{C}$ is similarly defined as the full monoidal, stable $\infty$-category of $\mathcal{SH}^{\mathcal{C},S^1}(k)$ generated by the $n$-fold tensor product $\mathbb{G}_\mathcal{C}^n$. We denote this colocalizing category by $\mathcal{SH}^{\mathcal{C},S^1}(k)(n)$ and the corresponding colocalization functor by $f_n^{\mathcal{C}}$. As usual, this gives rise to a tower of colocalization functors:
$$\cdots \to f^{\mathcal{C}}_{n+1}\to f_n^\mathcal{C}\to \cdots\to f_1^\mathcal{C}\to f_0^\mathcal{C}=1$$
defining the following cofiber sequences:
 $$f^{\mathcal{C}}_{n+1}\to f_n^\mathcal{C}\to s_n^\mathcal{C}\text{ and } f^{\mathcal{C}}_{1}\to f_0^\mathcal{C}\to f^{\mathcal{C}}_{0/n}.$$
 \subsection{Slices of cancellative correspondences}
 \begin{prop}\label[prop]{model for connective cover for canc corr}
 Suppose $\mathcal{C} $ is a category of correspondence over $k$ satisfying cancellation \textup{[\cite{bachmann2021cancellation}, Definition 2.11]}. Then, for every $n\geq 0$, there is a canonical equivalence:
    $$f_{n}^{\mathcal{C}}\simeq (-)^{\mathbb{G}_\mathcal{C}^n}\otimes \mathbb{G}_\mathcal{C}^n.$$    
    \end{prop}
\begin{proof}
    Let $\esc{D}\in \mathcal{SH}^{\mathcal{C},S^1}(k)(n)$ i.e., there is $\esc{E}\in \mathcal{SH}^{\mathcal{C},S^1}(k)$ such that $\esc{D}\simeq \esc{E}\otimes \mathbb{G}_\mathcal{C}^n.$ Then for every $\esc{F}\in \mathcal{SH}^{\mathcal{C},S^1}(k)$ we have:
    \begin{flalign*}
        \mathrm{Map}_{\mathcal{SH}^{\mathcal{C},S^1}(k)(n)}\big (\esc{D}, \esc{F}^{\mathbb{G}_\mathcal{C}^n}\otimes \mathbb{G}_\mathcal{C}^n\big)&= \mathrm{Map}_{\mathcal{SH}^{\mathcal{C},S^1}(k)(n)}\big (\esc{E}\otimes\mathbb{G}_\mathcal{C}^n, \esc{F}^{\mathbb{G}_\mathcal{C}^n}\otimes \mathbb{G}_\mathcal{C}^n\big)\\ &= \mathrm{Map}_{\mathcal{SH}^{\mathcal{C},S^1}(k)}\big (\esc{E}\otimes\mathbb{G}_\mathcal{C}^n, \esc{F}^{\mathbb{G}_\mathcal{C}^n}\otimes \mathbb{G}_\mathcal{C}^n\big)\\
        &\simeq  \mathrm{Map}_{\mathcal{SH}^{\mathcal{C},S^1}(k)}\big (\esc{E}, \esc{F}^{\mathbb{G}_\mathcal{C}^n}\big) \text{ (by } [\textup{\cite{bachmann2021cancellation}}, \text {Proposition 2.13])} \\
         &\simeq \mathrm{Map}_{\mathcal{SH}^{\mathcal{C},S^1}(k)}\big (\esc{E}\otimes{\mathbb{G}_\mathcal{C}^n}, \esc{F}\big)\\
         &= \mathrm{Map}_{\mathcal{SH}^{\mathcal{C},S^1}(k)}\big (\esc{D}, \esc{F}\big).
    \end{flalign*}
\end{proof}
Before we describe the slice connectivity for motives of generalized correspondences, we first need to set up an appropriate $t$-structure on the category of generalized mixed motivic spectra $\mathcal{SH}^{\mathcal{C},S^1}(k)$. As usual, the key will be a standard connectivity theorem for the functor that adds $\mathcal{C}$-transfers. In the rest of the section, unless otherwise specified, we let $k$ stand for a perfect field. 
\begin{prop}\label[prop]{connectivity for gen cor}
 Let $k$ be a perfect field. Let $\mu: Cor_k^{fr}\to \mathcal{C}$ be a generalized category of correspondences, and denote by $U: Sm_k\to \mathcal{C}$ the composite $\mu\circ \gamma$. Then the functor:
$$U_*:=\gamma_*\mu^*:\mathcal{SH}^{\mathcal{C},S^1}(k)\to \mathcal{SH}_{S^1}(k)$$
\begin{enumerate}
    \item preserves small limits and colimits and is conservative.
\end{enumerate}
Consequently, the functor $$U_*U^*: \mathcal{SH}_{S^1}(k)\to \mathcal{SH}_{S^1}(k)$$ 
    of addding $\mathcal{C}$-transfers has the following properties: 
    \begin{enumerate}[resume]
    \item It preserves small colimits and hence is a left adjoint.
    \item It preserves connectivity.
    \end{enumerate}
\end{prop} 
\begin{proof}
(1). First, since $U_*$ is a right adjoint, it preserves small limits. That (stably) $\gamma_*$ preserves colimits and is conservative is [\cite{elmanto2021motivic}, Proposition 3.5.2.], whereas the case of $\mu^*$ follows from stabilization of [\cite{bachmann2021cancellation}, Proposition 2.6(3)].

(2).  $U^*$, being a left adjoint, preserves colimits. So the claim follows from (1).

 (3).   Recall that the full subcategory of connective spectras in $\mathcal{SH}^{S^1}(k)$ is generated under colimits by motivic models of $\Sigma_+^\infty X$ for smooth schemes $X/k$ [\cite{elmanto2021motivic}, Proposition 3.1.13]. Thus, by (2) it suffices to show that $U_*U^*\Sigma_+^\infty X$ is a connective $S^1$-spectra over $k$.  By definition, we are looking at $\gamma_*\mu^*\mu_*\gamma^*\Sigma ^\infty X_+$. Since (unstably) $\gamma^*,\mu_*$ are left adjoints and $\mu^*$ preserves colimits [\cite{MR4324462}, Proposition 2.6 (3)] it follows that
 \begin{flalign*}     \gamma_*\mu^*\mu_*\gamma^*\Sigma ^\infty X_+&\simeq \gamma_*\mu^*\mu_*\Sigma_{fr} ^\infty \gamma^*X_+\\&\simeq \gamma_*\mu^*\Sigma_{\mathcal{C}} ^\infty \mu_*\gamma^*X_+\\&\simeq
   \gamma_*  \Sigma ^\infty _{fr}\mu^*\mu_*\gamma^*X_+
 \end{flalign*}
Since $Spc^{fr}(k)$ is semiadditive, this ccan be identified as
$$ \simeq \gamma_*  \mathrm{B} ^\infty _{fr}(\mu^*\mu_*\gamma^*X_+)^{gp}.$$
 By [\cite{elmanto2021motivic}, Lemma 3.5.4.] and the fact that (unstably) $\gamma_*$ preserves sifted colimits [\cite{elmanto2021motivic}, Proposition 3.2.15] the above is further equivalent to$$\mathrm{B}^\infty_{mot}  \gamma_* (\mu^*\mu_*\gamma^*X_+)^{gp}\simeq \mathrm{B}^\infty_{mot} (h^\mathcal{C}X_+)^{gp}.$$
  But the stable motivic connectivity of Morel tells us that $\mathrm{B}_{mot} ^\infty (h^\mathcal{C}X_+)^{gp}$ ($\simeq L_{mot}^{sp}(\mathrm{B}_{nis} ^\infty (h^\mathcal{C}X_+)^{gp})$ is connective. 
\end{proof}
This yields a canonical homotopy $t$-structure on $\mathcal{SH}_{S^1}^\mathcal{C}(k)$:
\begin{thm}\label{cancellative homotopy t structure}
    Let $U: Sm_k\to\mathcal{C}$ be a category of correspondences over a field $k$. Then there is a $t$-structure on $\mathcal{SH}^{\mathcal{C},S^1}(k)$ whose nonnegative part is the full subcategory spanned by $$\{\mal{E}\in\mathcal{SH}^{\mathcal{C},S^1}(k)\mid U_*\mal{E}\in \mathcal{SH}^{S^1}(k)_{\geq 0}\}. $$
\end{thm}
\begin{proof}
 By \Cref{connectivity for gen cor}(1) the full subcategory spanned by $$\{\mal{E}\in\mathcal{SH}^{\mathcal{C},S^1}(k)\mid U_*\mal{E}\in \mathcal{SH}^{S^1}(k)_{\geq 0}\}$$ is closed under colimits and extensions. Since it is the smallest such category containing $\Sigma^\infty _{\mathcal{C}}h^\mathcal{C}X_+$ for $X\in Sm_k$, the claim follows from [\cite{lurie2017higher}, Proposition 1.4.4.11.]. The compatibility follows from the 
\end{proof}
We call this the homotopy $t$-structure of $\mathcal{C}$-spectra.
\begin{rem}
  The standard $t$-structure on $\mathcal{SH}^{nis}_{S^1}(k)$ arises because $\mathcal{P}_{nis}(k)$ is an $\infty$-topos. However, being semiadditive, $\mathcal{P}_{nis}(\mathcal{C})$ is never an $\infty$-topos, so it does not, by itself, induce a $t$-structure on $\mathcal{SH}^{\mathcal{C},S^1}(k)$. 
\end{rem}
This $t$-structure admits a standard generative description:
\begin{prop}\label[prop]{monoidal nature of homotopy t structure}
    The homotopy $t$-structure of $\mathcal{C}$-spectra (or rather its non-negative part) is generated under colimits by $\Sigma^{\infty}_{\mathcal{C}}h_\mathcal{C}X_+\simeq U^*\Sigma^\infty X_+$. Consequently, the homotopy $t$-structure of $\mathcal{C}$-spectra is monoidally compatible \textup{[\cite{lurie2017higher}, Example 2.2.1.3]}.
\end{prop}
\begin{proof}
    Let $\mathscr{D}$ be the full subcategory of $\mathcal{SH}^{\mathcal{C},S^1}(k)$ generated under colimits by $\Sigma^{\infty}_{\mathcal{C}}h_\mathcal{C}X_+$ for $X\in Sm_k$. From the proof of \Cref{connectivity for gen cor}(3), we know that $$U_*\Sigma^{\infty}_{\mathcal{C}}h^{\mathcal{C}}X_+\simeq \mathrm{B}_{mot}^{\infty}(h^{\mathcal{C}}X_+)^{gp}\in \mathcal{SH}^{S^1}(k)_{\geq 0}.$$ It follows that $\mathscr{D}\subset \mathcal{SH}^{\mathcal{C},S^1}(k)_{\geq 0}$.

    To see the converse, we first note that, by [\cite{lurie2017higher}, Proposition 1.4.4.11], $\mathscr{D}$ is the non-negative part of a $t$-structure. The above observation shows that $U_*$ is right $t$-exact for this new $t$-structure as well. Since $U_*$ is a right adjoint, it follows immediately (for this particular $t$-structure) that $U_*$ is in fact $t$-exact. 

   Let $\mathcal{E}\in \mathcal{SH}^{\mathcal{C},S^1}(k)_{\geq 0}$, and consider the $t$-decomposition of $\mal{E}$ with respect to this new $t$-structure: $$t^\mathcal{D}_{<0}\mal{E}\to \mal{E}\to t^{\mathcal{D}}_{\geq0}\mal{E}.$$ Applying the $t$-exact functor $U_*$ and using that $\mathcal{E}\in \mathcal{SH}^{\mathcal{C},S^1}(k)_{\geq 0}$, we see that $U_*t^\mathcal{D}_{<0}\mal{E}\simeq 0$, yielding the equivalence $$U_* \mal{E}\to U_*t^{\mathcal{D}}_{\geq0}\mal{E}.$$ Since $U_*$ is conservative (\Cref{connectivity for gen cor}(1)), this implies $$ \mal{E}\simeq t^{\mathcal{D}}_{\geq0}\mal{E}\in\mathcal{D} .$$
\end{proof}
From the proof, it follows that:

\begin{cor}
    $U_*$ is $t$-exact for the homotopy $t$-structures.
\end{cor}

We shall now state and prove the slice connectivity theorem for a cancellative category of correspondences. Before we do that, we go back to the category of ordinary spectras on $Sm_k$ for a moment:
\begin{lem}\label[lem]{Gm power connectivity}
 Let $k$ be a perfect field and $\mal{E}$ be a Nisnevich-connected motivic $S^1$-spectrum on $Sm_k$. Then $\mal{E}^{\mathbb{G}}$ is connected. 
\end{lem}
\begin{proof}
    We assume $\mal{E}$ is given by an $\Omega$-spectrum $(\mal{E}_0,\mal{E}_1,\cdots)$. Since $\mal{E}$ is connected, each $\mal{E}_i$ is a Nisnevich-connected motivic space. It is well known that the $\Omega$-spectrum of $\mal{E}^{\mathbb{G}}$ is $\big((\mal{E}_0)^\mathbb{G},(\mal{E}_1)^\mathbb{G},\cdots\big)$. Thus, for every $i\leq 0$, we have $\pi_i^{st}(\mal{E}^\mathbb{G})\cong \pi_{0}^{nis}((\mal{E}_{-i})^\mathbb{G})$. Since $\mal{E}_{-i}$ is connected for all $i\leq 0$, the claim follows from [\cite{morel2012a1}, Theorem 6.13].\end{proof}
\begin{lem}
   Let $\mathcal{C}$ be a cancellative category of correspondences over a perfect field $k$. Then $f^\mathcal{C}_n$ is right $t$-exact.
\end{lem}

\begin{proof}
    Let $\mal{E}\in \mathcal{SH}^{\mathcal{C},S^1}(k)_{\geq 1}$. By \Cref{model for connective cover for canc corr}, we have for every $X\in Sm_k$
    \begin{flalign*}
        [U^* (X_+), \mal{E}^{\mathbb{G}_\mathcal{C}^n}]_{\mathcal{SH}^{\mathcal{C},S^1}(k)}&\simeq [U^* (X_+)\otimes{\mathbb{G}_\mathcal{C}^n}, \mal{E}]_{\mathcal{SH}^{\mathcal{C},S^1}(k)}\\
    &\simeq [U^* (X_+\wedge\mathbb{G}^n), \mal{E}]_{\mathcal{SH}^{\mathcal{C},S^1}(k)}\\
    &\simeq [X_+\wedge\mathbb{G}^n, U_*\mal{E}]_{\mathcal{SH}_{S^1}(k)}\\
    &\simeq [X_+, (U_*\mal{E})^{\mathbb{G}^n}]_{\mathcal{SH}_{S^1}(k)}
    \end{flalign*} 
    It follows that the Nisnevich sheafification of the presheaf
    $$X\mapsto [U^* (X_+), \mal{E}^{\mathbb{G}_\mathcal{C}^n}]_{\mathcal{SH}^{\mathcal{C},S^1}(k)}$$ is identical to $\pi_i^{st,nis}(U_*\mal{E})^{\mathbb{G}^n}$. By definition, $U_*\mal{E}\in \mathcal{SH}^{S^1}(k)_{\geq 1}$, i.e., $$\pi_i^{st,nis}(U_*\mal{E})=0\text{ for all }i\leq 0.$$ Therefore, by \Cref{Gm power connectivity}, $$\pi_i^{st,nis}(U_*\mal{E})^{\mathbb{G}^n}\simeq0 \text{ for all }i\leq 0.$$ In other words, $$(-)^{\mathbb{G}_\mathcal{C}^n}: \mathcal{SH}_{S^1}^\mathcal{C}(k)\to \mathcal{SH}_{S^1}^\mathcal{C}(k)$$ is right $t$-exact.

 Since $\mathbb{G}^n_{\mathcal{C}}\in \mathcal{SH}^{\mathcal{C},S^1}(k)_{\geq -1}$, \Cref{monoidal nature of homotopy t structure} implies that the functor $(-)^{\mathbb{G}_\mathcal{C}^n}\otimes \mathbb{G}_\mathcal{C}^n$ is right $t$-exact.\end{proof}

\begin{cor}\label[cor]{cancellative slice connectivity}
 Let $\mathcal{C}$ be a cancellative category of correspondences over a perfect field $k$. Then, for every $n\geq 0$, the functors $f_n^\mathcal{C}$, $f_{0/n}^\mathcal{C}$ and $s_n^\mathcal{C}$ are right $t$-exact as endo functors of $\mathcal{SH}^{\mathcal{C},S^1}(k)$.
\end{cor}
\underline{\textbf{Examples}}

\textbf{Slice connectivity of Voevodsky motives.}
We now turn our attention to the connectivity property of the slice filtration of Voevodsky motives. Since Voevodsky's category of finite correspondences $Cor^{ft}(k)$ is a category of correspondences in the sense of [\cite{MR4324462}, Definition 2.1] and has cancellation when $k$ is perfect [\cite{MR2804268}, Corollary 4.10], we can apply the results of the previous section to obtain slice connectivity for the larger presentable stable $\infty$-category of unbounded motivic complexes with finite transfers: $$\mathrm{DM}^{eff}(k)\simeq \mathcal{SH}_{S^1}^{ft}(k):=L_{\mathbb{A}^1}L_{nis}\mathcal{P}(Cor^{ft}(k),Spt).$$

However, we shall recast it here for $\mathrm{DM}_-^{eff}(k)$ for readers familiar with classical Voevodsky's theory of mixed motives. The adaptation of the slice filtration to $\mathrm{DM}_-^{eff}(k)$ first appeared in [\cite{MR2249535}]. We shall rewrite their definition to fit this paper's presentation. 

As usual, we define the $n$-th Tate-twisted category of Voevodsky motives as $$\mathrm{DM}_-^{eff}(k)(n):=\mathrm{DM}_-^{eff}(k)\wedge \mathbb{Z}^{tr}T^{\wedge n}\simeq \mathrm{DM}_-^{eff}(k)\wedge \mathbb{Z}(n).$$

As in \S3, the inclusion $\mathrm{DM}_-^{eff}(k)(n)\subset \mathrm{DM}_-^{eff}(k)$ is colocalizing. We denote the colocalization functor by $$f_n^{voe}:\mathrm{DM}_-^{eff}(k)\to \mathrm{DM}_-^{eff}(k).$$

Rather than a consequence, the following formula appears as the definition of Tate connective covers in [\cite{MR2249535}]. 
\begin{prop}\label[prop]{voe tate trunc formula}
    Let $k$ be a perfect field. Then for every $n\geq 0$, there is a canonical equivalence:
    $$f_{n}^{voe}\simeq (-)^{\mathbb{Z}(n)}\otimes \mathbb{Z}(n).$$    
\end{prop}
\begin{proof}
  By [\cite{MR2804268}, Corollary 4.10], we know that $Cor^{ft}_k$ satisfies cancellation. Thus, the claim follows from the same methods as in \Cref{model for connective cover for canc corr}.
\end{proof}
\begin{cor}
    Let $k$ be a perfect field. Then $f_n^{voe}$ is right $t$-exact for the standard homotopy $t$-structure on $\mathrm{DM}_-^{eff}(k)$.
\end{cor}
\begin{proof}
    This follows from using the formula for $f_n^{voe}$ as above (\Cref{voe tate trunc formula}), along with [\cite{MR2249535}, Lemma 2.1 and Lemma 2.3]. 
\end{proof}
\begin{cor}
    Let $k$ be a perfect field. Then $f_{0/n}^{voe}$ and $s_n^{voe}$ are right $t$-exact with respect to the standard homotopy $t$-structure on $\mathrm{DM}_-^{eff}(k)$.
\end{cor}

\textbf{Slice connectivity of Milnor Witt motives.} Instead of rewriting the entire content for Milnor-Witt motivic complexes, we mention it as a single theorem:
\begin{thm}
    Let $k$ be an infinite perfect field of characteristic $\neq 2$, and consider the stable $\infty$-category:
    $$\widetilde{\mathrm{DM}}^{eff}(k)\simeq \mathcal{SH}^{mw}_{S^1}(k):=\mathcal{SH}^{S^1}(\widetilde{Cor}(k)).$$
    Then for every $n\geq 0$, the inclusion $\widetilde{\mathrm{DM}}^{eff}(k)\otimes \widetilde{\mathbb{Z}}(n)\subset \widetilde{\mathrm{DM}}^{eff}(k) $ admits a right adjoint given by
    $$f_{n}^{mw}\simeq (-)^{\widetilde{\mathbb{Z}}(n)}\otimes\widetilde{\mathbb{Z}}(n),$$
    which is right $t$-exact for the standard homotopy $t$-structure on  $\widetilde{\mathrm{DM}}^{eff}(k)$ \textup{[\cite{MR4950840}, Corollary 3.46]}.
\end{thm}
\begin{proof}
    This follows from the general methods set up in this section (\Cref{cancellative homotopy t structure}, \Cref{cancellative slice connectivity}, etc.), along with the observation that in the given situation, the category $\widetilde{Cor}(k)$ of Milnor-Witt correspondences satisfies cancellation [\cite{MR4950840}, Theorem 4.26]. 
\end{proof}

\textbf{Slice connectivity of framed motivic spectras.} Similar to the previous one, we shall write down the slice filtration and its connectivity property for $S^1$-spectras on the category of framed correspondences (recall that its cancellation property is established in [\cite{elmanto2021motivic}, Theorem 3.5.8]). 
\begin{thm}\label{fr slice connectivity}
     Let $k$ be a perfect field. Then for every $n\geq 0$ the inclusion $\mathcal{SH}^{fr}_{S^1}(k)\otimes\mathbb{G}^n_{fr}\subset \mathcal{SH}^{fr}_{S^1}(k)$ is colocalizing with colocalization functor given by $$f^{fr}_n\simeq (-)^{\mathbb{G}_{fr}^n}\otimes \mathbb{G}_{fr}^n$$ which is right $t$-exact for the homotopy $t$-structure on $\mathcal{SH}^{fr}_{S^1}(k)$.
\end{thm}
\textbf{Slice connectivity of finite flat motivic spectras.} 
Recall that $Cor^{flf}$ satisfies cancellation over perfect fields [\cite{bachmann2021cancellation}, Theorem 3.5] and hence:
\begin{thm}\label{flf slice connectivity}
     When $k$ is a perfect field and $n\geq 0$, the inclusion $\mathcal{SH}^{flf}_{S^1}(k)\otimes\mathbb{G}^n_{flf}\subset \mathcal{SH}^{flf}_{S^1}(k)$ is colocalizing, with colocalization functor given by $$f^{flf}_n\simeq (-)^{\mathbb{G}_{flf}^n}\otimes \mathbb{G}_{flf}^n$$, which is right $t$-exact for the homotopy $t$-structure on $\mathcal{SH}^{flf}_{S^1}(k)$.
\end{thm}

\section{The infinite loop spaces for the slice filtration}
We conclude by adapting the recognition and reconstruction principles from [\cite{elmanto2021motivic}] to the slice filtration. Finally, combining the results of this section with those of the previous one, we shall establish the remaining connectivity property, namely that $f_n:\mathcal{SH}^{eff}(k)\to\mathcal{SH}^{eff}(k) $ preserves connectivity for every $n\geq 1$.
\subsection{Slices of the \texorpdfstring{$S^1$}{}-recognition principle} The Motivic $S^1$-recognition principle says that:
\begin{thm}[[\cite{elmanto2021motivic}, Proposition 3.1.9, Corollary 3.1.15\text{]}]
The adjunction$$\mathrm{B}^\infty_{nis}: \mathrm{CMon}(\mathcal{P}_{nis}(k))\leftrightarrows \mathcal{S}pt^{S^1}_{nis}(k): \Omega^\infty_{S^1}$$restricts to an equivalence of $\infty$-categories$$\mathrm{CMon}_{\mathrm{mot}}(\mathcal{H}^{\mathbb{A}^1}(k))^{\mathrm{gp}} \xrightarrow{\sim} \mathcal{SH}^{S^1}(k)_{\ge 0},$$identifying connective motivic $S^1$-spectra (on the right) with grouplike commutative monoids admitting $\mathbb{A}^1$-local deloopings (on the left). When $k$ is a perfect field:
\begin{itemize}
    \item $\mathcal{SH}^{S^1}(k)_{\ge 0}$ is the full subcategory generated under colimits by the motivic localization of smooth schemes.
    \item $\mathrm{CMon}_{\mathrm{mot}}(\mathcal{H}^{\mathbb{A}^1}(k))^{\mathrm{gp}}$ spans precisely the grouplike, $\mathbb{A}^1$-local commutative monoids whose $\mathbb{A}^1$-homotopy path components sheaf $\pi_0^{\mathbb{A}^1}$ is strongly $\mathbb{A}^1$-invariant.
\end{itemize}
\end{thm}
\begin{defn}
Let $\esc{G}$ be a Nisnevich local commutative monoid space.
\begin{enumerate}
    \item We say that $\esc{G}$ is strictly $n$-sphericalizable if for all $i$, $\pi_i^{\mathbb{A}^1}\esc{G}$ is strictly $\mathbb{A}^1$-invariant and $(\pi_{i}^{\mathbb{A}^1}\esc{G})_{-n-1}=0$.
    \item We say that $\esc{G}$ is a strictly $n$-spherical motivic monoid if $\mathrm{B}_{nis}^i\esc{G}$ is an $n$-spherical motivic space for all $i\geq 0$.
    \item We say that $\esc{G}$ is a strictly $L^{p,n}$-local motivic monoid if $\mathrm{B}_{nis}^i\esc{G}$ is an $L^{p,n}$-local motivic space for all $i\geq 0$.
    \end{enumerate}
\end{defn}\begin{rem}
    As usual, $\esc{G}$ is strictly $n$-spherical if and only if it is strictly $L^{2n+1,n+1}$-local.
\end{rem}
\begin{lem}\label[lem]{group complete reduction}
     Let $\esc{G}$ be a commutative monoid. Then $\esc{G}$ is strictly $L^{p,n}$-local if and only if $\esc{G}^{gp}$ is so.
\end{lem}
\begin{proof}
    Since $\Omega$ preserves local objects, the $L^{p,n}$ locality of $\mathrm{B}_{nis}\esc{G}$ implies that $\esc{G}^{gp}:=\Omega (\mathrm{B}_{nis}\esc{G})$ is $L^{p,n}$-local. Thus, the only-if part is clear.
    
    Conversely, suppose $\esc{G}^{gp}$ is strictly $L^{p,n}$ local motivic. By definition, this means that $\mathrm{B}^{i}\esc{G}^{gp}$ is $L^{p,n}$ local for all $i\geq 0$. For $i\geq 1$, we have $$\mathrm{B}^{i}\esc{G}^{gp}\simeq \mathrm{B}^{i}\Omega\mathrm{B}_{nis}\esc{G}\simeq \mathrm{B}_{nis}^i\esc{G}.$$ It remains to show that $\mathrm{B}^{0}\esc{G}\simeq \esc{G}$ is $L^{p,n}$ local. This follows from the fiber sequence $\esc{G}\to *\to \mathrm{B}_{nis}\esc{G}$, whose base and total spaces are given to be $L^{p,n}$ local.
\end{proof}
We first establish the relations among various strict $L^{p,n}$-localities. To do so, we introduce a standard generalization of [\cite{asok2023p}, Corollary 3.1.21].
\begin{lem}\label[lem]{pi_i of Lpq}
    If a pointed $\esc{X}$ is $S^{p,q}$-null, then for all $i\geq \max({p-q,1})$, $(\pi_i^{nis}\esc{X})_{-q}=0$.
\end{lem}
\begin{proof}
    We shall borrow the methods from [\cite{asok2023p}]. First, by $S^{p,q}$-connectivity, we may assume that $\esc{X}$ is connected. Indeed, let $\esc{X}_0$ be the connected component of the given base point. The morphism $\esc{X}_0 \to \esc{X}$ factors through $L^{p,q}\esc{X}_0$ because $\esc{X}$ is $S^{p,q}$-null. Since $L^{p,q}\esc{X}_0$ is connected by the $S^{p,q}$-connectivity theorem [\cite{asok2023p}, Corollary 3.1.26], the map $L^{p,q}\esc{X}_0\to \esc{X}$ factors through $\esc{X}_0$. We have constructed a retraction $$\esc{X}_0 \to L^{p,q}\esc{X}_0\to \esc{X}_0,$$ which proves that $\esc{X}_0$ is $S^{p,q}$-local. Now, for $i\geq 1$, we know that $\pi_i^{nis}\esc{X}=\pi_i^{nis}(\esc{X}_0)$, so we are done by appeal to [\cite{asok2023p}, Corollary 3.1.21]. 
\end{proof}

\begin{prop}\label[prop]{p independent strict Lpn}
    Let $\esc{G}$ be a commutative motivic monoid over a perfect field and let $p\geq n$ be arbitrary integers. Then the following are equivalent:
    \begin{enumerate}
        \item $\esc{G}$ is strictly $(n-1)$-spherical.
        \item $\esc{G}$ is strictly $L^{p,n}$-local.
        \item $\esc{G}$ is $L^{n,n}$-local and $(\pi_0^{nis}\esc{G})_{-n}=0$.
        \item $\esc{G}$ is $S^{n-1}_{\mathbb{A}}$-local and $(\pi_i^{nis}\esc{G})_{-n}=0$ for all $i\leq n-1$.
        \item $\esc{G}$ is $S^{p,n}$-local and $(\pi_i^{nis}\esc{G})_{-n}=0$ for all $i\leq p-n$.
        \item  $(\pi_i^{nis}\esc{G})_{-n}=0$ for all $i\geq 0$..
    \end{enumerate}
   
\end{prop}
\begin{proof}
    It suffices to prove the equivalence of (2), (5), and (6). 

    Assume (2). By definition $\esc{G}$ is $S^{p,n}$ local. So to show (2) $\implies$ (5) It remains to show that  $(\pi_i^{nis}\esc{G})_{-n}=0$ for all $i\leq p-n$. But, $\pi_i^{nis}\esc{G}\simeq \pi_{i+p-n+1}^{nis}\mathrm{B}_{nis}^{p-n+1}\esc{G}$. Since $\mathrm{B}_{nis}^{p-n+1}\esc{G}$ is $L^{p,n}$-local, the requirement follows from \Cref{pi_i of Lpq}.
    
    (5) $\implies$ (6) follows easily from \Cref{pi_i of Lpq}.

Finally, suppose (6) holds. Since group completion commutes with $\pi_0^{nis}$ and induces isomorphisms on higher $\pi_i$'s, using \Cref{group complete reduction} we may assume that $\esc{G}$ is grouplike, i.e., $\esc{G}\simeq \Omega \mathrm{B} \esc{G}$, so that $$\pi_i^{nis}\esc{G}\simeq \pi_{k+i}^{nis}\mathrm{B}_{nis}^k\esc{G} $$ for all $k\geq 0$. 

Now let $j\geq 1$. Then $\esc{M}=\mathrm{B}^j_{nis}\esc{G}$ has the property that for all $i\geq n$ the contractions $$ (\pi_{i}^{nis}\mathrm{B}_{nis}^j\esc{G})_{-n} =(\pi_{i-j}^{nis}\esc{G})_{-n}=0.$$ Since $\esc{M}$ is pointed and connected, [\cite{asok2023p}, Corollary 3.1.21] implies that $\esc{M}$ is $L^{p,n}$-local. To see that $\esc{G}$ is also $L^{p,n}$-local, we may use the fiber sequence (due to the group completeness of $\esc{G}$) $$\esc{G}\to *\to \mathrm{B}_{nis}\esc{G}$$ in which the base is $L^{p,n}$-local. This concludes (2).
\end{proof}
\begin{rem}\label[rem]{strong mot mon at pi0}
    Recall that, the condition in the proposition above, namely that $\esc{G}$ is a commutative motivic monoid over a perfect field, is equivalent to $\esc{G}$ being a $\mathbb{A}^1$-local commutative monoid with strictly $\mathbb{A}^1$-invariant $\pi_0^{\mathbb{A}^1}\esc{G}$ [\cite{elmanto2021motivic}, Theorem 3.1.2 (2)].
\end{rem}

Let us denote by $\mathrm{CMon}_{f_{0/n}}(\mathcal{H}^{\mathbb{A}^1}(k))$ the full subcategory of strictly $(n-1)$-spherical motivic commutative monoids. Similarly, denote by $\mathrm{CMon}_{L^{p,n}}(\mathcal{H}^{\mathbb{A}^1}(k))$ the full subcategory of strictly $L^{p,n}$-local motivic commutative monoids. Based on the results above, for all $p\geq n$ we have a canonical equivalence: 
$$\mathrm{CMon}_{f_{0/n}}(\mathcal{H}^{\mathbb{A}^1}(k))\simeq \mathrm{CMon}_{L^{p,n}}(\mathcal{H}^{\mathbb{A}^1}(k)).$$ 

Since $\Omega^{\infty}_{S^1}$ preserves $L^{p,n}$-local objects by construction, the adjunction $$\mathrm{B}^\infty_{nis}: \mathrm{CMon}(\mathcal{P}_{nis}(k))\leftrightarrows \mathcal{S}pt^{S^1}_{nis}(k): \Omega^\infty_{S^1}$$
descends to an adjunction $$\mathrm{B}^\infty_{f_{0/n}}:=f_{0/n}\mathrm{B}^\infty: L^{p,n}\mathrm{CMon}(\mathcal{H}^{\mathbb{A}^1}(k)) \leftrightarrows \mathcal{SH}^{S^1}(k)/f_n: \Omega^\infty_{S^1}.$$ Over $\mathrm{CMon}_{L^{p,n}}(\mathcal{H}^{\mathbb{A}^1}(k))$ we have $${\mathrm{B}^\infty_{f_{0/n}}}_{|\mathrm{CMon}_{n}(\mathcal{H}^{\mathrm{S}^{n}_\mathbb{A}})}\simeq \mathrm{B}^\infty_{|\mathrm{CMon}_{L^{p,n}}(\mathcal{H}^{\mathbb{A}^1}(k))}.$$ Since the functor $$\Omega^\infty_{S^1}: \mathcal{SH}_{\geq 0}\to \mathrm{CMon}(\mathcal{P}_{nis})$$ is fully faithful, so is $$\Omega^\infty_{S^1}: {(\mathcal{SH}^{S^1}/f_n)}_{\geq 0}\to L^{p,n}\mathrm{CMon}(\mathcal{H}^{\mathbb{A}^1}(k)).$$ Its essential image consists of the grouplike commutative $n$-spherical motivic monoids whose deloopings are all $n$-spherical motivic spaces. By definition, this is exactly $\mathrm{CMon}_{L^{p,n}}\big(\mathcal{H}^{\mathbb{A}^1}(k)\big)^{gp}$. Therefore, we obtain:
\begin{prop}\label[prop]{slice recognition}
Over a field, the motivic $S^1$-recognition equivalence (\textup{[\cite{elmanto2021motivic}, Proposition 3.1.9.]}) restricts to an equivalence: $$\mathrm{B}^\infty_{nis}\simeq \mathrm{B}^{{\infty}}_{f_{0/n}}:\mathrm{CMon}_{ L^{p,n}}(\mathcal{H}^{\mathbb{A}^1}(k))^{gp}\leftrightarrows {\mathcal{SH}^{S^1}(k)/f_{n}}: \Omega^\infty_{S^1}$$ where the left-hand side consists of strictly $L^{p,n}$-local commutative grouplike monoids, while $(\mathcal{SH}^{S^1}(k)/f_n)_{\geq 0}$ is the full subcategory of $\mathcal{SH}^{S^1}(k)/f_n$ consisting of connective spectra.
\end{prop}
\begin{cor}\label[cor]{slice recognition over perfect}
    Over a perfect field, there is an equivalence 
    $$\mathrm{B}^\infty_{{nis}} : \mathrm{CMon}_{n}(\mathcal{H}^{\mathbb{A}^1}(k))^{{gp}} \leftrightarrows \mathcal{SH}^{S^1}(k) / f_{n} : \Omega^\infty_{S^1}$$
    where the left-hand side consists of $\mathbb{A}^1$-invariant Nisnevich local grouplike commutative monoids $X$ whose $\pi_0^{\mathbb{A}^1} X$ is strongly $\mathbb{A}^1$-invariant and $(\pi_i^{\mathbb{A}^1} X)_{-n} = 0$ (i.e., $\pi_i^{\mathbb{A}^1} X \in {Ab}^{\mathbb{A}^1}_k$), and the right-hand side is the full subcategory of $\mathcal{SH}^{S^1}(k)$ generated under colimits by $\Sigma^\infty_+ X / f_n$.
\end{cor}
\begin{proof}
    By the preceding proposition, it remains to identify the two sides. The identification of the left-hand side (this step requires perfectness) follows from \Cref{p independent strict Lpn} and \Cref{strong mot mon at pi0}, while that of the right-hand side follows from \Cref{non negative slice generation}.
\end{proof}

\begin{cor}
    Over a perfect field, there is an equivalence  $$\mathrm{B}^\infty_{nis}:\mathrm{CMon}_{n}(\mathcal{H}^{\mathbb{A}^1}(k))^{gp}\leftrightarrows {\mathcal{SH}^{S^1}(k)/f_{n}}: \Omega^\infty_{S^1}$$
    where the left hand side consists of $\mathbb{A}^1$-invariant Nisnevich local grouplike commutative monoids $\esc{X}$ whose $\pi_0^{\mathbb{A}^1}\esc{X}$ is strongly $\mathbb{A}^1$ invariant and $(\pi_i^{\mathbb{A}^1}\esc{X})_{-n}=0$ (i.e.,  $\pi_i^{\mathbb{A}^1}\esc{X}\in Ab^{\mathbb{A}^1}_k$) and the right and side is the full subcategry of $\mathcal{SH}^{S^1}(k)$ generated under colimits by $\Sigma^\infty_+X/f_n$.
\end{cor}
\begin{proof}
    Using the proposition above, it remains to identify the left and right hand sides. The identification of the left hand side comes from \Cref{p independent strict Lpn}. The case of the right hand side comes from \Cref{non negative slice generation}.
\end{proof}

Recall from \Cref{characterizing n bir spectra} that:
\begin{prop}
    We have that $\mathcal{SH}^{S^1}(k)/f_n$ is the kernel (in $(Cat_\infty)_{*/}$) of the functor $$\underset{i\in \mathbb{Z}}{\prod} (\pi_i^{sp,nis})_{-n}:\mathcal{SH}^{S^1}\to Ab_k^{\mathbb{A}^1}.$$ It follows that $(\mathcal{SH}^{S^1}(k)/f_n)_{\geq 0}$ is the kernel of the functor $$\underset{i\in \mathbb{N}}{\prod} (\pi_i^{sp,nis})_{-n}:\mathcal{SH}^{S^1}_{\geq 0}\to Ab_k^{\mathbb{A}^1}.$$
\end{prop}
We may therefore compare the recognition principle by the following morphism of sequences of pointed $\infty$ categories:
\[
\xymatrix{
0\ar[r]&\mathcal{SH}^{S^1}(k)/f_n\ar[d]^{\Omega^\infty} \ar@{^{(}->}[r]&\mathcal{SH}^{S^1}(k)\ar[d]^{\Omega^\infty} \ar[rrr]^{\underset{i\in \mathbb{Z}}{\prod} (\pi_i^{sp,nis})_{-n}}&&&Ab_k^{\mathbb{A}^1}\ar@{=}[d]\\
&\mathrm{CMon}(\mathcal{H}^{p,n}(k))^{gp}\ar@{^(->}[r]&\mathrm{CMon}(\mathcal{H}^{\mathbb{A}^1}(k))^{gp}\ar[rrr]^-{\underset{i\in \mathbb{N}}{\prod} (\pi_i^{nis})_{-n
}}&&&Ab_k^{\mathbb{A}^1}
}
\]
\begin{cor}
When $k$ is a perfect field, the above diagram restricts to the following morphism of exact sequences, with all vertical adjunctions being equivalences:
\[
\xymatrix{
0\ar[r]&(\mathcal{SH}^{S^1}(k)/f_n)_{\geq 0}\ar@<.2em>[d]^{\Omega^\infty} \ar@{^{(}->}@<.2em>[r]&\mathcal{SH}^{S^1}(k)_{\geq 0}\ar@<.2em>[d]^{\Omega^\infty}\ar@<.2em>[l]^{L_{bir}^{n,sp}}\ar[rrr]^{\underset{i\in \mathbb{N}}{\prod} (\pi_i^{sp,nis})_{-n}}&&&Ab_k^{\mathbb{A}^1}\ar@{=}[d]\\
0\ar[r]&\mathrm{CMon}_{L^{p,n}}(\mathcal{H}^{p,n}(k))^{gp}\ar@<.2em>[u]^{\mathrm{B}_{nis}^\infty}\ar@{^{(}->}@<.2em>[r]&\mathrm{CMon}_{mot}(\mathcal{H}^{\mathbb{A}^1}(k))^{gp}\ar@<.2em>[l]^{L_{n}}\ar@<.2em>[u]^{\mathrm{B}_{nis}^\infty}\ar[rrr]^-{\underset{i\in \mathbb{N}}{\prod} (\pi_i^{nis})_{-n}}&&&Ab_k^{\mathbb{A}^1}
}
\]
Moreover, in this case, the first two categories in the top horizontal row are the respective subcategories generated under colimits by localizations of smooth schemes.
\end{cor}
One may similarly define spaces to be strictly $n$-birational if all $\mathrm{B}^i_{nis}$ are $n$-birational, and let $\mathrm{CMon}_{n}(\mathcal{H}^{n})$ be the full subcategory of strictly $n$-birational motivic commutative monoids. The adjunction $$\mathrm{B}^\infty_{nis}: \mathrm{CMon}(\mathcal{P}_{nis})\leftrightarrows \mathcal{SH}_{S^1}: \Omega^\infty_{S^1}$$ descends to an adjunction $$\mathrm{B}^\infty_{n}:=L_{bir}^{n,sp}\mathrm{B}^\infty: L^{n}_{bir}\mathrm{CMon}(\mathcal{H}^{n}) \leftrightarrows \mathcal{SH}^n(k): \Omega^\infty_{S^1}.$$ Over $\mathrm{CMon}_{n}(\mathcal{H}^{n})$ we have $${\mathrm{B}^\infty_{n}}_{|\mathrm{CMon}_{n}(\mathcal{H}^{n})}\simeq \mathrm{B}^\infty_{|\mathrm{CMon}_{n}(\mathcal{H}^{n})}.$$ Since the functor $$\Omega^\infty_{S^1}: \mathcal{SH}_{\geq 0}\to \mathrm{CMon}(\mathcal{P}_{nis})$$ is fully faithful, so is $\Omega^\infty_{S^1}: \mathcal{SH}^{n}_{\geq 0}\to \mathrm{CMon}(\mathcal{H}^{n})$, with essential image consisting of those grouplike commutative $n$-birational motivic monoids all of whose deloopings are $n$-birational motivic spaces. But this, by definition, is exactly $\mathrm{CMon}_{n}(\mathcal{H}^{n})^{gp}$. Therefore, we obtain:

\begin{prop}\label[prop]{birational recognition}
Over a perfect field, the motivic recognition equivalence restricts to an equivalence:    $$\mathrm{B}^\infty_{nis}\simeq \mathrm{B^{\infty}_{n}}: \mathrm{CMon}_{n}(\mathcal{H}^{n})^{gp}\leftrightarrows \mathcal{SH}^{n}_{\geq 0}: \Omega^\infty$$
where the left-hand side consists of strictly $n$-birational commutative grouplike spaces, while $\mathcal{SH}^n(k)_{\geq 0}$ is the full subcategory of $\mathcal{SH}^n(k)$ generated under colimits by $L_{bir}^{n,sp}\Sigma_+^{\infty}X$ for $X\in Sm_k$.
\end{prop}
It follows from the two recognition principles stated above that:
\begin{cor}
    A grouplike commutative motivic monoid $\esc{B}$ is strictly $n$-birational iff for all $i$, $(\pi_i^{nis}\esc{B})_{-n-1}=0$ iff $\esc{B}$ is strictly $L^{p,n}$-local. In other words, $\mathrm{CMon}_{n}(\mathcal{H}^{n})^{gp}$ is the kernel of the morphism $$\underset{i\in \mathbb{N}}{\prod} (\pi_i^{nis})_{-n-1}:\mathrm{CMon}_{mot}(\mathcal{H}^{\mathbb{A}^1})^{gp}\to Ab_k^{\mathbb{A}^1}.$$
\end{cor}
\begin{rem}
 We expect that a version of the above holds for connected spaces (without even an $A_\infty$ structure), but we are unable to establish this at the moment. 
\end{rem}
We note the following result as a side remark:
\begin{lem}
 A group-like commutative monoid $\esc{G}$ on $k$ is a strictly $L^{1,1}$-local motivic monoid iff $\esc{G}$ is birational local. In particular, $Ab^{\mathbb{A}^1}_k/f_1\simeq Ab_k^b$, where the right-hand side denotes birational sheaves of abelian groups. \textup{(This last equivalence was obtained in [\cite{0bat}, Lemma 3.3.9.].)}
\end{lem}
\begin{proof}
By the corollary above, this follows immediately from [\cite{0bat}, Corollary 3.2.9].\end{proof}
\begin{ex}
    By the corollary above and \Cref{gw/f1=z} it follows that the reflection of $\underline{\mathrm{GW}}$ under the inclusion $Ab^b_k\subset Ab^{\mathbb{A}^1}_k$ is given by $\mathbb{Z}$.
\end{ex}
\subsection{Slices of the $\mathbb{P}^1$-loop space machine}
In this final section, we aim to analyze the behavior of the motivic recognition principle under the slice filtration. As we have been doing so far, we shall first recall the ordinary motivic reconstruction theorem for motivic $\mathbb{P}^1$-spectra:
\begin{prop}[[{\cite{elmanto2021motivic}}, Corollary 3.5.9, Theorem 3.5.14(i)\text{]},]\label[prop]{reconstruction theorem}
    Let $k$ be a perfect field. Then there exist monoidal equivalences of stable $\infty$-categories: 
    \begin{flalign}
    \mathcal{SH}^{fr,S^1}(k)\underset{\sigma^\infty_{fr}}{\simeq}\mathcal{SH}^{fr,eff}(k)\underset{\gamma_*}{\simeq} \mathcal{SH}^{eff}(k)
  \label{reconstruct}
    \end{flalign} 
    In fact, the composite equivalence is an equivalence of $t$-categories with the standard homotopy $t$-structures on both sides.
\end{prop}

The main point of this subsection is to establish the following theorem.
\begin{thm}\label{slice reconst}
   Let $k$ be a perfect field. Then for every $n\geq 0$ there are equivalences of stable $\infty$-categories:
   \begin{flalign}
       \mathcal{SH}_{S^1}^{fr}(k)\otimes\mathbb{G}^n_{fr}\simeq \mathcal{SH}^{eff}(k)\wedge \mathbb{G}^n\label{tate fr equiv}\\
       \mathcal{SH}_{S^1}^{fr}(k)/f_n^{fr}\simeq \mathcal{SH}^{eff}(k)/f_n\label{tate trunc equiv}
   \end{flalign}
\end{thm}
\begin{proof}
    Since the (composite) equivalence of \Cref{reconstruction theorem} is a monoidal equivalence that maps $\mathbb{G}$ to $\mathbb{G}_{fr}$ and vice versa, it follows that the restriction of this equivalence to $n$-connective covers induces the first equivalence.
    The second equivalence then follows from the first and from \Cref{reconstruction theorem} itself.
\end{proof}
\begin{cor}[compare with \Cref{effective slice exact}(2),(3); see also \Cref{bachmann gm conservativity}]\label[cor]{eff fn preserves conn}
  Let $k$ be a perfect field and $n\geq 0$ a natural number. Then
\begin{enumerate}
    \item The colocalization functor $f_{n}:\mathcal{SH}^{eff}(k)\to \mathcal{SH}^{eff}(k)$ is $t$-exact.
     \item The functor $s_{n}:\mathcal{SH}^{eff}(k)\to \mathcal{SH}^{eff}(k)$ is right $t$-exact (of amplitude $0$)..
     \item The functor $f_{0/n}:\mathcal{SH}^{eff}(k)\to \mathcal{SH}^{eff}(k)$ is right $t$-exact. \textup{(We have already obtained this in \Cref{effective slice exact}(1), though using the slice conjectures and the $S^1$-analog established in \S2.4.)}.
\end{enumerate}  
\end{cor}
\begin{proof}
    (1) The right $t$-exactness of $f_n$ follows from \Cref{slice reconst}(\cref{tate fr equiv}) and \Cref{fr slice connectivity}. Since $f_n$ is a colocalization, using the orthogonality condition, it follows that it is left $t$-exact too.
    
   Using their defining cofiber sequences for $s_n$ and $f_{0/n}$, (2) and (3) then follow easily from the right $t$-exactness of $f_n$ from (1).
\end{proof}
And then, using the same techniques as in \Cref{slice conn for p1 spectra}, we have:
\begin{thm}\label{p1 fn preserves conn}
 Over a perfct field, for every $n\geq 0$, the functor $f_n: \mathcal{SH}^{\mathbb{P}^1}(k)\to \mathcal{SH}^{\mathbb{P}^1}(k)$ is $t$-exact for the homotopy $t$-structure. Consequently, $s_n, f_{0/n}$ are right $t$-esact.
\end{thm}

This concludes our quest for the connectivity property of the slice filtration sought in this paper. To complete the story outlined in this section, we shall describe the behavior of the motivic $\mathbb{P}^1$-recognition principle under the slice filtration. For this, we first note an immediate corollary of the last statement in the corollary above.
\begin{rem}\label[rem]{shveff}
    Recall that  $\mathcal{SH}^{veff}(k)$ is generally defined as the full subcategory of $\mathcal{SH}^{\mathbb{P}^1}(k)$ generated under colimits and extensions by smooth schemes. However, [\cite{MR3743071}, Remark after Proposition 4] shows that the full subcategory generated under colimits is already closed under extensions. Hence, $\mathcal{SH}^{veff}(k)$ is the full subcategory of $\mathcal{SH}^{\mathbb{P}^1}(k)$ generated under colimits by smooth schemes. On the other hand, it is thus the non-negative part of the effective homotopy $t$-structure.
\end{rem}
 
\begin{cor}\label[cor]{SHveff/fn}
Let $k$ be a perfect field and let $n\geq 0$ be a natural number. Then the functor $f_{0/n}:\mathcal{SH}^{eff}(k)\to \mathcal{SH}^{eff}(k)$ restricts to a functor $f_{0/n}:\mathcal{SH}^{veff}(k)\to \mathcal{SH}^{veff}(k)$. Moreover, the essential image, denoted $\mathcal{SH}^{veff}(k)/f_n$, is the full subcategory of $\mathcal{SH}^{eff}/f_n$ generated by $f_{0/n}\Sigma_+^\infty X$ for $X\in Sm_k$ and is precisely the positive part of the $t$-structure on $\mathcal{SH}^{eff}/f_n$ described in \textup{\Cref{effective slice exact} (1)}.
\end{cor}
\begin{proof}
The first line is a restatement of \Cref{eff fn preserves conn}(3).

   Let $\mathscr{C}$ be the full subcategory of $\mathcal{SH}^{eff}/f_n$ generated by $f_{0/n}\Sigma_+^\infty X$ for $X\in Sm_k$. Since $f_{0/n}$ preserves colimits (being a left adjoint), it follows that we have a colimit closed inclusion $\mathcal{SH}^{veff}/f_n\subset \mathscr{C}$. On the other hand, by construction we know that $f_{0/n}\Sigma_+^\infty X\in \mathcal{SH}^{veff}(k)/f_n.$
\end{proof}

To be able to prove the slice version of the recognition principle, we shall restate its ordinary motivic version in a clear step-by-step format.
\begin{thm}[[{\cite{elmanto2021motivic}}, Theorem 3.5.14(i)\text{]}]\label{ordinary recognition theorem}
    Let $k$ be a perfect field. The recognition equivalence is the composition of the following equivalences:
    \begin{enumerate}
        \item $\mathcal{H}^{fr}(k)^{gp}\simeq \mathcal{SH}^{S^1,fr}(k)_{\geq 0}$
        \item $\mathcal{SH}^{S^1,fr}(k)_{\geq 0}\simeq \mathcal{SH}^{eff,fr}(k)_{\geq 0}$
        \item $\mathcal{SH}^{eff,fr}(k)_{\geq 0}\simeq \mathcal{SH}^{eff}(k)_{\geq 0}$
        \item $\mathcal{SH}^{eff}(k)_{\geq 0}\simeq \mathcal{SH}^{veff}(k)$.
    \end{enumerate}
\end{thm}
\begin{proof}
    (1) Topologically, the left-hand side requires all deloopings to be $\mathbb{A}^1$-invariant, but this is automatic for framed motivic spaces [\cite{elmanto2021motivic}, Corollary 3.4.13].
    
    (2) and (3) restrict the reconstruction equivalences (\Cref{reconstruct}) to their connective covers. 

   In (4), the left-hand side is generated under colimits by smooth schemes, whereas the right-hand side is generated under colimits and extensions by smooth schemes. Their equality has already been addressed in \Cref{shveff}.
\end{proof}

In order to ease the treatment of slice detection in the recognition principle, we introduce $L^{n,n}\mathcal{H}^{fr}(k)$ as the localization of $\mathcal{H}^{fr}(k)$ at the projection $\gamma^*(\mathbb{G}^n_+\to S^0)$. Similar to \Cref{slice quotient is Lpn}, it is evident that after stabilization, one has,
\begin{lem}\label[lem]{framed slice=Lnn}
    $\mathcal{SH}^{fr,S^1}(k)/f_n\simeq L^{n,n}\mathcal{SH}^{fr,S^1}(k):=\mathrm{Stab}(L^{n,n}\mathcal{H}^{fr}(k)).$
\end{lem}

\begin{lem}\label[lem]{defn of strong Lnn framed motiv}
    A framed motivic space $\esc{M}$ lies in $L^{n,n}\mathcal{H}^{fr}(k)$ if and only if $\gamma_*\esc{M}$ lies in $L^{n,n}\mathcal{H}^{\mathbb{A}^1}(k)$.
\end{lem}
\begin{proof}
    This follows from the definition and the adjunction $\gamma^*\dashv \gamma _*$.
\end{proof}

Analogously, let us denote by $\mathcal{H}^{fr}_{L^{n,n}}(k)$ the full subcategory consisting of framed motivic spaces whose framed deloopings are $L^{n,n}$-local. Also, let $\mathcal{H}^{fr}_n(k)$ be the full subcategory of $L^{n,n}\mathcal{H}^{fr}(k)$ consisting of framed motivic spaces $\esc{X}$ such that $$\pi_0^{nis}(\gamma_*\esc{X})_{-n}=0.$$ 
\begin{prop}\label[prop]{identifying strong framed slice}
    There is a canonical equivalence
    $$\mathcal{H}^{fr}_{L^{p,n}}(k)\simeq\mathcal{H}^{fr}_{n}(k).$$
    This is equivalent to the full subcategory of $\esc{Y}\in \mathcal{H}^{fr}(k)$ consisting of those for which, for all $i\geq 0$ $$(\pi_i^{nis}(\gamma_*\esc{Y}))_{-n}=0.$$
\end{prop}
\begin{proof}
    Since $\gamma_*$ commutes with the bar construction, using \Cref{defn of strong Lnn framed motiv} the claim follows easily from the characterization of strictly $L^{n, n} $ local commutative motivic monoids with commutative motivic monoids having $\mathbb A^1$ homotopy groups with trivial $n$ fold $\mathbb{G}_m$ contraction (\Cref{p independent strict Lpn}(6)). 
\end{proof}
\begin{rem}
    It is easy to reformulate the above results in terms of $L^{p,n}$ localizations, as established in the non-framed case in \S5.1.
\end{rem}

We are now ready to state the recognition principle for the slice quotient:
\begin{thm}\label{p1 recognition for slices}
    Let $k$ be a perfect field. Then the motivic recognition equivalence restricts to an equivalence:
    $$\gamma_*\Sigma_{fr}^{\infty}:\mathcal{H}^{fr}_n(k)^{gp}\leftrightarrows\mathcal{SH}^{veff}(k)/f_n:\Omega^\infty_\mathbb{P}\gamma^*$$ 
    where  the lefthand is the full subcategory of $\esc{Y}\in \mathcal{H}^{fr}(k)$ such that for all $i\geq 0$ $(\pi_i^{nis}(\gamma_*\esc{Y}))_{-n}=0$ while the right hand side is the full subcategory of $\mathcal{SH}^{eff}(k)/f_n$ generated under colimits by smooth schemes.
\end{thm}
\begin{proof} 
    The equivalence,
    $$\mathcal{H}_{L^{n,n}}^{fr}(k)^{gp}\simeq L^{n,n}\mathcal{SH}^{fr}_{S^1}(k)_{\geq 0}$$  is a topological result, following from \Cref{ordinary recognition theorem}(1) where the left hand is the full subcategory of $\mathcal{H}^{fr}(k)^{gp}$ admitting $\gamma^*(\mathbb{G}_{+}^n\to S^0)$ local delooping. So \Cref{identifying strong framed slice} settles the left-hand side.   
    
 On the other hand,
    \begin{flalign*}
        L^{n,n}\mathcal{SH}^{fr,S^1}(k)_{\geq 0}&\simeq (\mathcal{SH}^{fr, S^1}(k)/f_{n}^{fr})_{\geq 0}\textup{ [by \Cref{framed slice=Lnn}]}\\
        &\simeq (\mathcal{SH}^{eff}(k)/f_n)_{\geq 0}\textup{ [by \Cref{tate trunc equiv}]}\\&\simeq \mathcal{SH}^{veff}(k)/{f}_n \textup{ [by \Cref{SHveff/fn}]}.
    \end{flalign*}
\end{proof}
\begin{rem}\label[rem]{rhs of p1 slice recognition}
   The category $\mathcal{SH}^{veff}(k)/f_n$ in the above theorem admits an equivalent description. Namely, it is the \textit{modified} generalized slice quotient of $\mathcal{SH}^{veff}(k)$. To describe it, first recall the generalized slice filtration [\cite{MR3743071}], generated by the monoidal subcategories $\mathcal{SH}^{veff}(k)\wedge T^n$. The modification required for this comparison is to instead consider the slice filtration generated by $\mathcal{SH}^{veff}(k)(n):=\mathcal{SH}^{veff}(k)\wedge \mathbb{G}^n$. Comparing with Bachmann's notation in [\cite{MR3743071}, \S4], we have:$$\mathcal{SH}^{veff}(k)(n)=\mathrm{SH}^{veff}(k)(n)[-n].$$ Calling the corresponding covering functors the generalized Tate and generalized Thom covering functors, it is clear that $\tilde{f}^{tate}_n\simeq \tilde{f}^{thom}_n\Sigma^{-n}$. Finally, by the modified generalized slice quotient, we mean the cofiber of $$\tilde{f}_n\Sigma^{-n}\simeq\tilde{f}^{tate}_n\to \tilde{f}^{tate}_0\simeq \tilde{f}_0,$$ denoting its image on $\mathcal{SH}^{veff}(k)$ by $\mathcal{SH}^{veff}(k)/\tilde{f}_n$. It follows from [Lemma 10] of loc. cit. that $\tilde{f}^{tate}_n\simeq f_n\circ (-)_{\geq 0}$. Consequently, restricting to $\mathcal{SH}^{veff}$ yields $\tilde{f}^{tate}_{0/n}\simeq f_{0/n}$. This gives the desired description:$$\mathcal{SH}^{veff}(k)/f_n\simeq \mathcal{SH}^{veff}(k)/\tilde{f}_n.$$So the slice recognition can be rewritten as:$$\mathcal{H}^{fr}_n(k)^{gp}\simeq\mathcal{SH}^{veff}(k)/\tilde{f}_n.$$
\end{rem}
 \begin{rem}
     It is easy to reformulate the recognition and reconstruction theorem for slices in terms of the birational localization, as in \S3.1. We leave the details to the reader. However, we must mention two important aspects of considering the $n$ birational localization over the $L^{n+1,n+1}$ localization. These are as follows: 
     \begin{enumerate}
     \item When $n=0$, one does not require the notion of strictly birational framed motivic spaces. This is because the $0$-birational localization functor commutes with the bar construction [\cite{0bat}, Corollary 3.2.4]. (We shall return to this in an upcoming work [\cite{bas}].) It is, a priori, not clear whether a similar statement holds for $L^{1,1}$-localization (or $L^{2,1}$-localization, for that matter).
     \item Although $\gamma_*$ preserves $L^{n+1,n+1}$-local objects almost by definition (\Cref{defn of strong Lnn framed motiv}), it is not clear whether it preserves $L^{n+1,n+1}$-local equivalences. On the other hand, [\cite{bachmann2019voevodsky}, Lemma 3.1(4)] shows that $\gamma_*$ not only preserves $n$-birational local objects but also $n$-birational local equivalences, thereby commuting with the $n$-birational localization functors.
     \end{enumerate}
      
 \end{rem}
 \begin{rem}
    Writing down a recognition principle for $f_n$ or $s_n$ is nontrivial, even in the $S^1$-stable setting, because, unlike localizations, the colocalization $f_n$ behaves poorly under stabilization. (Although, luckily, ${\Omega^\infty_{S^1}}_{\big|\mathcal{SH}^{S^1}(k)_{\ge 0}}$ and ${\Omega^\infty_{\mathbb{P}^1}}_{\big|\mathcal{SH}^{\mathrm{veff}}(k)}$ preserve $L^{n,n}$-colocal objects [\cite{asok2023p}, Proposition 3.2.12].) Meanwhile, $s_n$ is neither a localization nor a colocalization (of the $\infty$-category of $0$-tate coverings) unless $n=0$ (where $s_0 = f_{0/1}$, so that $\mathcal{H}^{fr}_1(k)^{gp}\simeq s_0\mathcal{SH}^{veff}(k)$).
 \end{rem}
 \begin{rem}
     Let us return to \Cref{bachmann gm conservativity} for a moment to summarize our limitations. As noted there, over infinite perfect fields, Bachmann shows (using the validity of a version of the slice conjectures) that $f_n$ (and, by extension, $f_{0/n}$ and $s_n$) preserves the connectivity of $ S^1$-Spectra. In this section, we have achieved this for effective spectra without appealing to the slice conjectures, but this method does not extend to $S^1$-spectra. On the other hand, while the methods of \S2 do apply to $S^1$-spectra (in fact, over all fields), they work only for $f_{0/n}$, not for $f_n$ or $s_n$. This is something we have not been able to achieve using either of the methods (\S2.4 and \S5.2) described in this paper.
 \end{rem}

\phantomsection
\bibliographystyle{amsalpha}	
\renewcommand\refname{Bibliography}
\bibliography{references}
\noindent\rule{\textwidth}{0.4pt} 
\end{document}